\documentclass[10pt]{article}
\usepackage[utf8]{inputenc} 
\usepackage[T1]{fontenc}
\usepackage{lmodern}

\usepackage{amsthm,amsmath,amsfonts,amssymb,stmaryrd,bbm}
\usepackage{mathtools}
\usepackage{enumitem}
\usepackage{mathrsfs} 
\usepackage{mathtools} 

\usepackage{subcaption} 

\usepackage{anyfontsize}

\usepackage{graphicx}
\usepackage[dvipsnames]{xcolor}
\usepackage{caption,tikz}

\usepackage[english]{babel}
\usepackage[colorlinks=true, linkcolor=blue, urlcolor=black, citecolor=blue,pdfstartview=FitH]{hyperref}

\numberwithin{equation}{section}

\let\originalleft\left
\let\originalright\right
\renewcommand{\left}{\mathopen{}\mathclose\bgroup\originalleft}
\renewcommand{\right}{\aftergroup\egroup\originalright}

\newlength{\bibitemsep}
\newlength{\bibparskip}
\let\oldthebibliography\thebibliography
\renewcommand\thebibliography[1]{\oldthebibliography{#1}
  \setlength{\parskip}{\bibitemsep}
  \setlength{\itemsep}{\bibparskip}}

\DeclareMathOperator{\im}{Im}
\DeclareMathOperator{\Tr}{Tr}
\DeclareMathOperator{\Var}{Var}
\DeclareMathOperator{\Cov}{Cov}

\DeclareMathOperator{\OO}{O}

\DeclareMathOperator{\Crt}{Crt}
\DeclareMathOperator{\Id}{Id}
\DeclareMathOperator{\diag}{diag}
\DeclareMathOperator{\supp}{supp}
\DeclareMathOperator{\Spec}{Spec}

\newcommand{\mc}[1]{\mathcal{#1}}
\newcommand{\mf}[1]{\mathfrak{#1}}
\newcommand{\ms}[1]{\mathscr{#1}}

\newcommand{\ii}{\mathrm{i}}

\newcommand{\defeq}{\vcentcolon=}
\newcommand{\eqdef}{=\vcentcolon}

\renewcommand{\epsilon}{\varepsilon}

\renewcommand{\leq}{\leqslant}
\renewcommand{\geq}{\geqslant}

\renewcommand{\P}{\mathbb{P}}
\newcommand{\E}{\mathbb{E}}
\newcommand{\R}{\mathbb{R}}
\newcommand{\C}{\mathbb{C}}
\newcommand{\N}{\mathbb{N}}

\newcommand{\abs}[1]{\left\lvert #1 \right\rvert}

\newcommand{\vertiii}[1]{{\left\vert\kern-0.25ex\left\vert\kern-0.25ex\left\vert #1 
    \right\vert\kern-0.25ex\right\vert\kern-0.25ex\right\vert}}
\newcommand{\ip}[1]{\left\langle #1 \right\rangle}
\newcommand{\diff}{\mathop{}\!\mathrm{d}}

\theoremstyle{plain} 
\newtheorem{thm}{Theorem}[section]

\newtheorem{lem}[thm]{Lemma}
\newtheorem{cor}[thm]{Corollary}

\newtheorem{assn}[thm]{Assumption}

\newtheorem{defn}[thm]{Definition}

\newtheorem{rem}[thm]{Remark}

\newtheorem*{notn}{Notation}

\makeatletter
\renewcommand{\section}{\@startsection
{section}
{1}
{0mm}
{-2\baselineskip}
{1\baselineskip}
{\normalfont\large\scshape\centering}} 
\makeatother

\makeatletter
\renewcommand{\subsection}{\@startsection
{subsection}
{2}
{0mm}
{-\baselineskip}
{0 \baselineskip}
{\normalfont\bf\itshape}} 
\makeatother

\makeatletter
\renewcommand{\subsubsection}{\@startsection
{subsubsection}
{3}
{0mm}
{-\baselineskip}
{0 \baselineskip}
{\normalfont\bf\itshape}} 
\makeatother

\def\author#1{\par
    {\centering{\authorfont#1}\par\vspace*{0.05in}}
}

\def\titlefont{\fontsize{13}{15}\bfseries\boldmath\selectfont\centering{}}
\def\authorfont{\fontsize{13}{15}}

\let\affiliationfont\rhfont

\def\address#1{\par
    {\centering{\affiliationfont#1\par}}\par\vspace*{11pt}
}

\def\title#1{
    \thispagestyle{plain}
    \vspace*{-14pt}
    \vskip 79pt
    {\centering{\titlefont #1\par}}%
    \vskip 1em
}

\begin{document}

~\vspace{-1.4cm}

\title{Complexity of Bipartite Spherical Spin Glasses}

\vspace{1cm}
\noindent

\begin{center}
\begin{minipage}[c]{0.33\textwidth}
 \author{Benjamin M\textsuperscript{c}Kenna}
\address{Courant Institute \\ 
   New York University \\
   E-mail: bmckenna@fas.harvard.edu}
 \end{minipage}
 \end{center}

\begin{abstract}
This paper characterizes the annealed complexity of bipartite spherical spin glasses, both pure and mixed. This means we give exact variational formulas for the asymptotics of the expected numbers of critical points and of local minima. This problem was initially considered by Auffinger-Chen \cite{AufChe2014}, who gave upper and lower bounds on this complexity. We find two surprising connections between pure bipartite and pure single-species spin glasses, which were studied by Auffinger-Ben Arous-\v{C}ern\'{y} \cite{AufBenCer2013}. First, the local minima of any pure bipartite model lie primarily in a low-energy band, similar to the single-species case. Second, for a more restricted set of pure $(p,q)$ bipartite models, the complexity matches exactly that of a pure $p+q$ single-species model.
\end{abstract}

\vspace*{0.05in}

\noindent \emph{Date:} March 21, 2023

\vspace*{0.05in}

\noindent \hangindent=0.2in \emph{Keywords and phrases:} Bipartite spin glasses, spherical spin glasses, landscape complexity, Kac--Rice formula, Dyson equation. 

\vspace*{0.05in}


{
	\hypersetup{linkcolor=black}
	\tableofcontents
}


\newpage
\section{Introduction}


\subsection{History and motivations.}\

Multi-species spin systems were first introduced in the 1970s in the physics of metamagnets \cite{KinCoh1975}, and in the last fifteen years, their development has been accelerated by applications of two kinds. First, in many social and biological networks it is natural to group individuals into two populations, and the result can be modelled with bipartite spin glasses, for example in immunology with two types of immune cells \cite{AglBarBarGalGueMoa2013}. Second, certain types of neural networks, such as Hopfield networks and restricted Boltzmann machines, can be mapped to bipartite spin systems \cite{BarGenGue2010, AglBarGalGueMoa2012, BarGenSolTan2018}.

Partially motivated by these applications, physical properties like the free energy of bipartite spin glasses have been studied, mostly for what we will later call $(1,1)$ models with Ising spins or variations thereof. These were treated both in the physics literature, first by Korenblit--Shender and Fyodorov--Korenblit--Shender \cite{KorShe1985, FyoKorShe1987A, FyoKorShe1987B} and later by Guerra and co-authors \cite{BarGenGue2011, BarGalGuePizTan2014}, both under the assumption of replica symmetry, and then by Hartnett \emph{et al.} assuming replica symmetry breaking \cite{HarParGei2018}; and in the mathematical literature, first as an upper bound due to Barra \emph{et al.} \cite{BarConMinTan2015} and then a matching lower bound due to Panchenko \cite{Pan2015}. The free energy for \emph{spherical} bipartite models was established by Auffinger and Chen at high temperature \cite{AufChe2014}, allowing for mixtures and small external fields, and by Baik and Lee at all temperatures other than some critical one \cite{BaiLee2020}, restricted to what we will call pure spherical $(1,1)$ models. Recently, bipartite spin glasses appeared as a model example in Mourrat's program to relate the free energy of disordered systems to infinite-dimensional Hamilton--Jacobi equations \cite{Mou2021}. 

Beyond applications, bipartite spin systems also serve as a toy model for spin glasses beyond the purely mean-field regime. Spins interact with each other in two groups, a waystation between the best-understood mean-field spin glasses (where all spins interact with each other on equal footing) and the eventual goal of spin glasses with nearest-neighbor interactions.


\subsection{Results.}\

In this paper, we study the \emph{complexity} of high-dimensional bipartite spherical models. That is, write $\mc{H}_N$ for an $N$-dimensional bipartite spin glass, which is a real-valued random function defined on a product of two high-dimensional spheres (see precise definitions in Section \ref{sec:bsg_complexity}). Write $\Crt^{\textup{tot}}_N(t)$ for the (random) number of critical points of $\mc{H}_N$ at which $\mc{H}_N \leq Nt$, and $\Crt^{\textup{min}}_N(t)$ for the number of such local minima. We wish to understand the large-$N$ asymptotics of $\frac{1}{N}\log\E[\Crt^{\textup{tot}}_N(t)]$ and $\frac{1}{N}\log\E[\Crt^{\textup{min}}_N(t)]$. 

This landscape-complexity program --- counting critical points of high-dimensional random functions to understand their geometry --- was initiated by Fyodorov \cite{Fyo2004} for a certain toy model of disordered systems, and re-discovered by Auffinger--Ben Arous--\v{C}ern\'{y} for spherical spin glasses \cite{AufBenCer2013, AufBen2013}. Complexity of spherical bipartite models was first studied by Auffinger and Chen \cite{AufChe2014}, who found continuous functions $J, K : \R \to \R$ such that
\[
    J(t) \leq \lim_{N \to \infty} \frac{1}{N} \log \E[\Crt^{\textup{min}}_N(t)] \leq K(t).
\]
Their strategy was to compare bipartite spin glasses with a coupled pair of usual (single-species) spin glasses. They also established that $J(t) > 0$ for some $t$, so that the system has positive complexity, and that $\lim_{t \to -\infty} K(t) = -\infty$, so that it makes sense to define the ``smallest zero of $K$'' which is thus a lower bound for the ground state.

In Theorem \ref{thm:bsg} below, we give exact formulas for $\lim_{N \to \infty} \frac{1}{N}\log\E[\Crt^{\textup{tot}}_N(t)]$ and $\lim_{N \to \infty} \frac{1}{N}\log\E[\Crt^{\textup{min}}_N(t)]$ that are of the form
\begin{equation}
\label{eqn:introresultform}
    \sup_{u \in \mf{D}} \left\{ \int_\R \log\abs{\lambda} \mu_\infty(u,\lambda) \diff \lambda - \frac{\|u\|^2}{2}\right\}.
\end{equation}
Here $\mf{D}$ is some subset of $\R$ (for pure models) or $\R^3$ (for mixtures), and the deterministic probability measures $\mu_\infty(u,\cdot)$ are found by solving a system of two coupled quadratic equations in two scalar unknowns. This system arises from the \emph{Matrix Dyson Equation (MDE)}, developed to describe the local eigenvalue behavior of large random matrices in \cite{AjaErdKru2019, ErdKruSch2019,AltErdKruNem2019}, which we describe below.

For the special case of pure $(p,q)$ models with ratio $\gamma = \frac{p}{p+q}$ (see definitions below), the measures $\mu_\infty(u,\cdot)$ are rescalings of the semicircle law, so these variational problems can be solved explicitly. The resulting complexity functions turn out to be the same as those describing the pure $p+q$ usual (single-species) spherical spin glass, as established by Auffinger--Ben Arous--\v{C}ern\'{y} \cite{AufBenCer2013}; see Corollary \ref{cor:explicit} below. This is surprising, since the models look quite different. It remains to be seen if this analogy holds for other types of critical points, such as saddle points, for bipartite models with ratios $\gamma$ other than $\frac{p}{p+q}$, or for more than two communities.

We also show that pure $(p,q)$ models, with any ratio $\gamma$, exhibit a band-of-minima phenomenon similar to pure spherical spin glasses. More precisely, there exists a threshold $-E_\infty(p,q,\gamma) < 0$ such that, with high probability and for any $\epsilon > 0$, all local minima have energy values below $N(-E_\infty(p,q,\gamma)+\epsilon)$; see Corollary \ref{cor:einfty} below. It would be interesting to understand the role of this threshold in, say, Langevin dynamics.


\subsection{The (Matrix) Dyson equation.}\

We now introduce the MDE and its role, in both random matrix theory in general, and this paper in particular, as well as commenting on some subtleties in the literature. For a full review we direct the reader to \cite{Erd2019}. In full generality, before any connection to random matrices, one considers the following problem: Given a von Neumann algebra $\ms{A}$ with identity $1$ and tracial state $\ip{\cdot}$, a self-adjoint element $A \in \ms{A}$, some $z$ in the complex upper half-plane, and a linear operator (the ``self-energy operator'') $\mc{S} : \ms{A} \to \ms{A}$ that preserves the cone of positive semidefinite elements of $\ms{A}$, one seeks the unique solution \cite{HelFarSpe2007} $M(z) \in \ms{A}$ to the constrained equation
\begin{equation}
\label{eqn:introvonneumannmde}
\begin{split}
	&1 + (z1 - A + \mc{S}[M(z)])M(z) = 0, \\
	&\text{subject to} \quad \im M(z) = \frac{M(z) - M^\ast(z)}{2\ii} > 0,
\end{split}
\end{equation}
which in this generality is called the \emph{Dyson equation}. If $\mc{S}$ is symmetric with respect to the inner product $\ip{x,y} \defeq \ip{x^\ast y}$, then in fact $\ip{M(z)}$ is the Stieltjes transform of a probability measure $\mu$ on $\R$ \cite{AltErdKru2020}: $\ip{M(z)} = \int \frac{\mu(\diff \lambda)}{\lambda - z}$. We say that the Dyson equation induces the measure $\mu$, or that $\mu$ is the measure coming from the Dyson equation.

The connection to random matrices comes via global laws. Suppose $H_N$ is an $N \times N$ random matrix, say real symmetric, with eigenvalues $\lambda_1(H_N) \leq \cdots \leq \lambda_N(H_N)$ and empirical spectral measure $\hat{\mu}_{H_N} = \frac{1}{N} \sum_{i=1}^N \delta_{\lambda_i(H_N)}$. We ask for a global law, i.e., a deterministic probability measure $\mu_\infty$ on $\R$ such that $\hat{\mu}_{H_N} \to \mu_\infty$ (say weakly in probability), or a little more generally a sequence $(\mu_N)_{N=1}^\infty$ of deterministic probability measures on $\R$ such that $d(\hat{\mu}_{H_N},\mu_N) \to 0$ (say in probability) for some distance $d$ metrizing weak convergence. For the most classical ensembles, global laws are of course well-known: take the semicircle distribution for Wigner matrices, or the Mar\v{c}enko-Pastur distribution for sample covariance matrices, for example.

For non-classical random matrices, however, it is not so obvious what measure to choose -- but one can think of the Dyson equation, for some good choice of $\ms{A}$, $A$, and $\mc{S}$, as a machine producing these measures. For example, suppose $H_N$ is a deformed Wigner matrix with variance profile, meaning $H_N$ is $N \times N$ and has the form $H_N = A_N + W_N$, where $A_N = \E[H_N]$ is deterministic, and $W_N$ has independent entries up to symmetry distributed as $(W_N)_{ij} \sim \sigma_{ij} \nu$ for some fixed probability measure $\nu$, centered with unit variance, and some constants $\sigma_{ij} > 0$. The special case $A_N = 0$ and $\sigma_{ij} \equiv 1/\sqrt{N}$ recovers the usual Wigner matrices. In this generality, one should instead consider, for each $N$, the Dyson equation over the von Neumann algebra $\ms{A} = \C^{N \times N}$, with the choices $A = A_N$ and $\mc{S}[T] = \mc{S}_N[T] = \E[W_NTW_N]$. Under mild conditions on the model, $H_N$ can be shown to satisfy a global law with respect to the sequence of measures $(\mu_N)_{N=1}^\infty$ so produced \cite{AjaErdKru2019}.\footnote{
Actually \cite{AjaErdKru2019}, like many related papers, proves a much harder local law. Roughly, this says that the measure from the Dyson equation describes, not just the collective behavior of all the eigenvalues, but also the behavior of just a few of them. In our work, we do not need the full strength of these types of results.
}

Auffinger and Chen \cite{AufChe2014} already observed that, via the Kac--Rice formula, understanding the complexity of bipartite spherical spin glasses reduces to studying the determinant of a deformed Gaussian matrix with a variance profile. The point of the current work is to connect this observation with our result, joint with Ben Arous and Bourgade in the companion paper \cite{BenBouMcK2022}, that
\begin{equation}
\label{eqn:intro_detcon}
	\lim_{N \to \infty} \left( \frac{1}{N} \log \E[\abs{\det(H_N)}] - \int_\R \log\abs{\lambda} \mu_N(\diff \lambda) \right) = 0
\end{equation}
for such matrices, where $\mu_N$ comes from a Dyson equation, under mild conditions, which we spend a good bit of this article showing are satisfied by most bipartite models. 

One subtle problem in going from \eqref{eqn:intro_detcon} to \eqref{eqn:introresultform} comes in trying to replace $\mu_N$ with some $\mu_\infty$ that has (a) nice properties and (b) as simple of a description as possible. This paper is the third in the series starting with \cite{BenBouMcK2022}; in the second, also joint with Ben Arous and Bourgade \cite{BenBouMcK2021II}, we study the complexity of two models (the ``elastic-manifold'' model and a general signal-plus-noise model) where the measures $\mu_\infty$ are ultimately free additive convolutions of the semicircle law with another probability measure, and thus very well understood. One high-level difference in this third paper is that the measures $\mu_\infty$ for bipartite spin glasses do not seem to admit such classical descriptions: ``found by solving two equations in two unknowns'' is the simplest we can find, and it is not clear if a fundamentally different description exists. To get down to two equations in two unknowns, it is important that multiple Dyson equations, even over different von Neumann algebras, may produce measures that give a good deterministic approximation for the same random matrix. One should weave carefully between these different perspectives to get a clean result. In Section \ref{subsec:goe} we explain how to do this in the toy case of the Gaussian Orthogonal Ensemble (GOE).

Finally, after a draft of this paper was posted on the arXiv, Kivimae \cite{Kiv2021} carried out a second moment computation for pure $(p,q)$ bipartite spin glasses with $\min(p,q) \geq 97$, showing that, at exponential scale, the second moment of the count of critical points in low sublevel sets matches the first squared (i.e., that ``quenched equals annealed''). In particular, this provides an upper bound for the ground state energy, matching the lower bound we give below.

The paper is organized as follows. In Section \ref{sec:bsg_complexity} we state our main results, both variational formulas for general models and closed-form formulas for the special case stated above. In Section \ref{sec:bsg} we give the proofs, which rely on determinant asymptotics for large random matrices as established in the companion paper \cite{BenBouMcK2022}, and strategies for applying these to complexity as established in the companion paper \cite{BenBouMcK2021II}. 


\subsection*{Notations.}\

We write $\|\cdot\|$ for the operator norm on elements of $\C^{N \times N}$ induced by Euclidean distance on $\C^N$, and $\|\cdot\|_{\textup{hs}}$ for the normalized Hilbert-Schmidt norm $\|T\|_{\textup{hs}}^2 = \frac{1}{N} \sum_{i,j} \abs{T_{ij}}^2$. For operators $\mc{S} : \C^{N \times N} \to \C^{N \times N}$, we write $\|\mc{S}\|$ for the operator norm induced by $\|\cdot\|$, and $\|\mc{S}\|_{\textup{hs} \to \|\cdot\|}$ for the norm satisfying $\|\mc{S}[T]\|_{\textup{hs}} \leq \|\mc{S}\|_{\textup{hs} \to \|\cdot\|} \|T\|$. We let
\[
    \|f\|_{\text{Lip}} = \sup_{x \neq y} \abs{\frac{f(x)-f(y)}{x-y}}
\]
for test functions $f : \R \to \R$, and consider the following two distances on probability measures on the real line (called bounded-Lipschitz and Wasserstein-$1$, respectively):
\begin{align*}
    d_{\textup{BL}}(\mu,\nu) = \sup\left\{\abs{\int_\R f \diff (\mu - \nu)} : \|f\|_{\text{Lip}} + \|f\|_{L^\infty} \leq 1\right\}, \quad {\rm W}_1(\mu,\nu) = \sup\left\{\abs{\int_\R f \diff (\mu - \nu)} : \|f\|_{\text{Lip}} \leq 1\right\}.
\end{align*}
We write $\ell(\mu)$ for the left edge (respectively, $r(\mu)$ for the right edge) of a compactly supported measure $\mu$. For an $N \times N$ Hermitian matrix $M$, we write $\lambda_{\min{}}(M) = \lambda_1(M) \leq \cdots \leq \lambda_N(M) = \lambda_{\max{}}(M)$ for its eigenvalues and 
\[
    \hat{\mu}_M = \frac{1}{N}\sum_{i=1}^N \delta_{\lambda_i(M)}
\]
for its empirical measure. We write $\odot$ for the entrywise (i.e., Hadamard) product of matrices. We slightly abuse notation by giving two meanings to $\diag$. First we use the usual one, where $\diag(a_1,\ldots,a_N)$ means the $N \times N$ diagonal matrix with given scalars $a_1, \ldots, a_N$ along the diagonal. Second, we use it for a function from matrices to matrices: if $T$ is a matrix, then $\diag(T)$ is the diagonal matrix of the same size obtained by setting all off-diagonal entries to zero. In equations, we sometimes identify diagonal matrices with vectors of the same size. We write $B_R(0)$ for the ball of radius $R$ about zero in the relevant Euclidean space. We use $(\cdot)^T$ for the matrix transpose, which should be distinguished both from $(\cdot)^\ast$ for the matrix conjugate transpose, and from $\Tr(\cdot)$ for the matrix trace.

Unless stated otherwise, $z$ will always be a complex number in the upper half-plane $\mathbb{H} = \{z \in \C : \im(z) > 0\}$, and we always write its real and imaginary parts as $z = E+\ii \eta$. 


\subsection*{Acknowledgements.}\
We wish to thank Tuca Auffinger for bringing the bipartite spin glass model to our attention, and G{\'e}rard Ben Arous, Paul Bourgade, Wei-Kuo Chen, Krishnan Mody, Jean-Christophe Mourrat, and Ofer Zeitouni for helpful discussions. We are also grateful to the referees for greatly improving the readability of the paper, to Mark Sellke and Brice Huang for help correcting some formulas in a previous version of this paper, and to Yan Fyodorov for pointing us towards additional references on bipartite spin glasses in the physics literature. This work was supported in part by NSF grant DMS-1812114.


\section{Main results}
\label{sec:bsg_complexity}

We follow the notation of \cite{AufChe2014}. If $M \in \N$, write $S^M$ for the $(M-1)$-sphere in $\R^M$ with radius $\sqrt{M}$. Fix some $\gamma \in (0,1)$, suppose that we decompose each positive integer $N \geq 2$ as $N = N_1 + N_2$, where $N_1$ and $N_2$ are positive integers satisfying $N_1 \approx \gamma N$ in the precise sense
\begin{equation}
\label{eqn:ratio_gamma}
    \frac{N_1-1}{N-2} = \gamma.
\end{equation}
(Notice the abuse of notation: $N_1$ is actually a sequence of positive integers.) For any $p, q \geq 1$, define the pure bipartite Hamiltonian for $u = (u_1, \ldots, u_{N_1}) \in S^{N_1}$ and $v = (v_1, \ldots, v_{N_2}) \in S^{N_2}$ as 
\[
    \mc{H}_{N,p,q}(u,v) = \sum_{1 \leq i_1, \ldots, i_p \leq N_1} \sum_{1 \leq j_1, \ldots, j_q \leq N_2} g_{i_1, \ldots, i_p, j_1, \ldots, j_q} u_{i_1} \ldots u_{i_p} v_{j_1} \ldots v_{j_q}
\]
where the $g$ variables are i.i.d. centered Gaussians with variance $N/(N_1^pN_2^q)$. Equivalently, $\mc{H}_{N,p,q}$ is the centered Gaussian process on $S^{N_1} \times S^{N_2}$ with covariance
\[
    \E[\mc{H}_{N,p,q}(u,v)\mc{H}_{N,p,q}(u',v')] = N \left( \frac{1}{N_1} \sum_{i=1}^{N_1} u_iu'_i \right)^p \left( \frac{1}{N_2} \sum_{i=1}^{N_2} v_iv'_i \right)^q.
\]
Notice that this interaction is genuinely bipartite, meaning that it is not a pure spin glass of the concatenated vector $(u_1, \ldots, u_{N_1}, v_1, \ldots, v_{N_2})$. Define the ``mixed'' Hamiltonian
\[
    \mc{H}_N(u,v) = \sum_{p, q \geq 1} \beta_{p,q} \mc{H}_{N,p,q}(u,v)
\]
where the nonnegative double sequence $(\beta_{p,q})_{p,q \geq 1}$ is not identically zero and decays fast enough; for example, $\sum_{p, q \geq 1} (1+\epsilon)^{p+q}\beta_{p,q}^2 < \infty$ for some $\epsilon > 0$ suffices. Define $\xi : [-1,1]^2 \to \R$ by
\[
    \xi(x,y) = \sum_{p,q \geq 1} \beta_{p,q}^2x^py^q
\]
assumed to be normalized as 
\[  
    \xi(1,1) = 1.
\]
We will say the model is ``pure $(p_0,q_0)$'' if $\beta_{p,q} = \delta_{pp_0}\delta_{qq_0}$, and ``pure'' if it is pure $(p_0,q_0)$ for some $p_0$, $q_0$. Define
\begin{align*}
    \xi'_1 &= \partial_x \xi(x,y)|_{x=y=1} = \sum_{p, q \geq 1} p\beta_{p,q}^2, && \xi''_1 = \partial_{xx}\xi(x,y)|_{x=y=1} = \sum_{p, q \geq 1} p(p-1)\beta_{p,q}^2, \\
    \xi'_2 &= \partial_y \xi(x,y)|_{x=y=1} = \sum_{p, q \geq 1} q\beta_{p,q}^2, && \xi''_2 = \partial_{yy}\xi(x,y)|_{x=y=1} = \sum_{p, q \geq 1} q(q-1)\beta_{p,q}^2, \\
    \xi''_{12} &= \partial_x \partial_y \xi(x,y)|_{x=y=1} = \sum_{p,q \geq 1} pq \beta_{p,q}^2.
\end{align*}
Since $\xi(1,1) = 1$, one can check with Cauchy--Schwarz that $\xi''_i + \xi'_i - (\xi'_i)^2 \geq 0$ for each $i = 1, 2$ (the details are carried out in \cite[Lemma 6.1]{BelCerNakSch2022}). Thus we may define 
\[
    \alpha_i = \sqrt{\xi''_i + \xi'_i - (\xi'_i)^2}.
\]
Notice that $\alpha_1 = \alpha_2 = 0$ if and only if the model is pure. To define the remaining quantities, we need the following lemma.
\begin{lem}
\label{lem:sigma_def}
We have
\begin{equation}
\label{eqn:covariance-variance}
    (\xi''_{12} - \xi'_1 \xi'_2)^2 \leq \alpha_1^2\alpha_2^2.
\end{equation}
\end{lem}
\begin{proof}
For any function $f(p,q)$, write $\E_\xi[f] = \sum_{p,q \geq 1} \beta_{p,q}^2 f(p,q)$; since $\xi(1,1) = 1$, this is a normalized expectation. In this notation, \eqref{eqn:covariance-variance} just reads $\Cov_\xi(p,q)^2 \leq \Var_\xi(p)\Var_\xi(q)$.
\end{proof}
Due to Lemma \ref{lem:sigma_def}, it makes sense to define
\begin{equation}
\label{eqn:def_sigma}
\begin{split}
    \sigma &= \begin{cases} \frac{\xi''_{12} - \xi'_1 \xi'_2}{\alpha_1\alpha_2} & \text{if } \alpha_1 \neq 0 \neq \alpha_2, \\ 0 & \text{otherwise,} \end{cases} \\
    \sigma_{\pm} &= \sqrt{1+\sigma} \pm \sqrt{1-\sigma}.
\end{split}
\end{equation}

\subsection{Results.}\

For $t \in \R$, write $\Crt^{\textup{tot}}_N(t)$ for the number of critical points of $\mc{H}_N$ at which $\mc{H}_N \leq Nt$, and $\Crt^{\textup{tot}}_N$ for the total number of critical points of $\mc{H}_N$. Write also $\Crt^{\textup{min}}_N(t)$ for the number of local minima of $\mc{H}_N$ at which $\mc{H}_N \leq Nt$, and $\Crt^{\textup{min}}_N$ for the total number of local minima of $\mc{H}_N$. In the statement of the main theorem, we will need the half-space
\[
	H_t = \{(u_0, u_1, u_2) : u_0 \leq t\} \subset \R^3.
\]

We will mostly restrict ourselves to these objects, but will occasionally consider, for a Borel set $B \subset \R$, the quantities $\Crt^{\textup{tot}}_N(B)$ (respectively, $\Crt^{\textup{min}}_N(B)$) for the number of critical points (respectively, number of local minima) of $\mc{H}_N$ at which $\mc{H}_N \in NB$, and the corresponding space $H_B = B \times \R^2 = \{(u_0,u_1,u_2) \in \R^3 : u_0 \in B\}$. These correspond to our previous notation as, for example, $\Crt^{\textup{min}}_N(t) = \Crt^{\textup{min}}_N((-\infty,t])$ and $H_t = H_{(-\infty,t]}$. 

The following assumption will be made throughout the paper, although we will only sometimes write it explicitly.

\begin{assn}
\label{assn}
Assume that the function $\xi$ satisfies
\begin{equation}
\label{eqn:nondegenerate}
	\xi''_1 > 0 \quad \text{and} \quad \xi''_2 > 0.
\end{equation}
This condition is satisfied if and only if the model is neither a linear combination of $(1,q)$ spins for different $q$ values, nor a linear combination of $(p,1)$ spins for different $p$ values.
\end{assn}

For the statement of the main theorem, Theorem \ref{thm:bsg}, we will need the following family of probability measures.

\begin{lem}
\label{lem:bsg_scalars}
For each $u = (u_0, u_1, u_2) \in \R^3$ and each $z \in \mathbb{H}$, there exists a unique solution $\{m_1(u,z),m_2(u,z)\} \in \C^2$ to the system 
\begin{align}
\label{eqn:bsg_scalars}
    \begin{cases} 1+\left(z-\frac{1}{\gamma}(\frac{\alpha_1\sigma_+}{2}u_1 + \frac{\alpha_1\sigma_-}{2}u_2 - \xi'_1u_0) + \frac{\xi''_1}{\gamma} m_1(u,z) + \frac{\xi''_{12}}{\gamma} m_2(u,z)\right) m_1(u,z) = 0, \\
    1+\left(z-\frac{1}{1-\gamma}(\frac{\alpha_2\sigma_-}{2}u_1 + \frac{\alpha_2\sigma_+}{2} u_2 - \xi'_2u_0) + \frac{\xi''_2}{1-\gamma} m_2(u,z) + \frac{\xi''_{12}}{1-\gamma} m_1(u,z)\right) m_2(u,z) = 0, \\
    \im(m_1(u,z)) > 0, \\
    \im(m_2(u,z)) > 0.\end{cases}
\end{align}
For each $u$, this unique solution has the property that $m(u,z) \defeq \gamma m_1(u,z) + (1-\gamma)m_2(u,z)$ is the Stieltjes transform of some probability measure $\mu_\infty(u)$ on $\R$, which is compactly supported with a bounded, H\"olderian density $\mu_\infty(u,\cdot)$.
\end{lem}

\begin{lem}
\label{lem:scontinuous}
For each $u \in \R^3$, define
\begin{equation}
\label{eqn:s_bsg}
	\mc{S}_{\textup{bsg}}[u] = \int_\R \log\abs{\lambda} \mu_\infty(u,\lambda) \diff \lambda - \frac{\|u\|_2^2}{2}.
\end{equation}
Then $\mc{S}_{\textup{bsg}}[u]$ is a continuous function of $u$ with $\lim_{\|u\| \to \infty} \mc{S}_{\textup{bsg}}[u] = -\infty$.
\end{lem}

\begin{thm}
\label{thm:bsg}
Suppose Assumption \ref{assn} is satisfied. Then
\begin{align}
\label{eqn:bsg_tot}
\begin{split}
    \Sigma^{\textup{tot}}(t) \defeq \lim_{N \to \infty} \frac{1}{N}\log \E[\Crt_N^{\textup{tot}}(t)] &= \frac{1+\gamma\log\left(\frac{\gamma}{\xi'_1}\right) + (1-\gamma)\log\left(\frac{1-\gamma}{\xi'_2}\right)}{2} + \sup_{u\in H_t} \mc{S}_{\textup{bsg}}[u], \\
    \Sigma^{\textup{tot}} \defeq \lim_{N \to \infty} \frac{1}{N}\log \E[\Crt_N^{\textup{tot}}] &= \frac{1+\gamma\log\left(\frac{\gamma}{\xi'_1}\right) + (1-\gamma)\log\left(\frac{1-\gamma}{\xi'_2}\right)}{2} + \sup_{u \in \R^3} \mc{S}_{\textup{bsg}}[u],
\end{split}
\end{align}
and (due to Lemma \ref{lem:scontinuous}) these suprema are achieved, possibly not uniquely.

Furthermore, define the set 
\begin{equation}
\label{eqn:bsg_defG}
    \mc{G} = \{u \in \R^3 : \mu_\infty(u)((-\infty,0)) = 0\}
\end{equation}
of $u$ values whose corresponding measures $\mu_\infty(u)$ are supported in the right half-line. Then $\mc{G}$ is convex and closed, and we have
\begin{align}
\label{eqn:bsg_min}
\begin{split}
    \Sigma^{\textup{min}}(t) &\defeq \lim_{N \to \infty} \frac{1}{N}\log \E[\Crt_N^{\textup{min}}(t)] = \frac{1+\gamma\log\left(\frac{\gamma}{\xi'_1}\right) + (1-\gamma)\log\left(\frac{1-\gamma}{\xi'_2}\right)}{2} + \sup_{u\in H_t \cap \mc{G}} \mc{S}_{\textup{bsg}}[u], \\
    \Sigma^{\textup{min}} &\defeq \lim_{N \to \infty} \frac{1}{N}\log \E[\Crt_N^{\textup{min}}] = \frac{1+\gamma\log\left(\frac{\gamma}{\xi'_1}\right) + (1-\gamma)\log\left(\frac{1-\gamma}{\xi'_2}\right)}{2} + \sup_{u \in \mc{G}} \mc{S}_{\textup{bsg}}[u],
\end{split}
\end{align}
and these again suprema are achieved, possibly not uniquely ($H_t \cap \mc{G}$ is nonempty for every $t$).
\end{thm}

\begin{rem}
\label{rem:pure}
In the special case when the model is pure $(p,q)$, the result simplifies somewhat: The $\alpha_i$ vanish in \eqref{eqn:bsg_scalars}, and with them all dependence on $u_1$ and $u_2$, so $\mu_\infty((u_0,u_1,u_2))$ is a function of $u_0$ only. Thus $\mc{G}$ takes the form 
\[
    \mc{G} = \{u_0 \times \R^2 : u_0 \in \mc{G}_{\textup{pure}}\}
\]
for some set $\mc{G}_{\textup{pure}} = \mc{G}_{\textup{pure}}(p,q,\gamma) \subset \R$. Since $\mc{G}$ is convex and closed, and \eqref{eqn:subsetofg} below shows that it contains points whose first coordinates are arbitrarily large and negative, in fact $\mc{G}_{\textup{pure}}$ must be an interval of the form
\begin{equation}
\label{eqn:gpure}
    \mc{G}_{\textup{pure}} = (-\infty,-E_\infty(p,q,\gamma)]
\end{equation}
for some $E_\infty(p,q,\gamma)$, which will turn out to be an important threshold. (This notation and sign convention is intended to evoke \cite{AufBenCer2013}; see the discussion below.)

One consequence of this simplification is that the variational problems for pure $(p,q)$ models are one-dimensional:
\begin{align*}
    \lim_{N \to \infty} \frac{1}{N}\log \E[\Crt_N^{\textup{tot}}(t)] &= \frac{1+\gamma\log\left(\frac{\gamma}{p}\right) + (1-\gamma)\log\left(\frac{1-\gamma}{q}\right)}{2} + \max_{u_0 \leq t} \mc{S}_{\textup{bsg}}[(u_0,0,0)], \\
    \lim_{N \to \infty} \frac{1}{N}\log \E[\Crt_N^{\textup{tot}}] &= \frac{1+\gamma\log\left(\frac{\gamma}{p}\right) + (1-\gamma)\log\left(\frac{1-\gamma}{q}\right)}{2} + \max_{u_0 \in \R} \mc{S}_{\textup{bsg}}[(u_0,0,0)],
\end{align*}
and similarly for minima.
\end{rem}

\begin{cor}
\label{cor:einfty}
For every pure $(p,q)$ model satisfying Assumption \ref{assn}, the quantity $-E_\infty(p,q,\gamma)$ defined in \eqref{eqn:gpure} is strictly negative, and most local minima have energy below $-NE_\infty(p,q,\gamma)$ in the following senses:
\begin{itemize}
\item For all $t \geq -E_\infty(p,q,\gamma)$, we have $\Sigma^{\textup{min}}(t) = \Sigma^{\textup{min}}(-E_\infty(p,q,\gamma))$.
\item For any $\epsilon > 0$, we have
\[
    \lim_{N \to \infty} \frac{1}{N}\log \P(\Crt^{\textup{min}}_N((-E_\infty(p,q,\gamma)+\epsilon,\infty)) \geq 1) = -\infty.
\]
\end{itemize}
\end{cor}

In the extra-special case of a pure $(p,q)$ model with $\gamma = \frac{p}{p+q}$, we can solve the variational problems explicitly, because then the relevant Hessian is (almost, up to small error) a generalized Wigner matrix and $\mu_\infty(u)$ is (exactly) a rescaled semicircle law. In the following we write the log-potential of semicircle as 
\begin{align*}
    \Omega(x) &= \int_{-2}^2 \log\abs{\lambda - x} \frac{\sqrt{4-\lambda^2}}{2\pi} \diff \lambda = \begin{cases} \frac{x^2}{4} - \frac{1}{2} & \text{if } \abs{x} \leq 2, \\ \frac{x^2}{4} - \frac{1}{2} - \left(\frac{\abs{x}}{4} \sqrt{x^2-4} - \log \left(\frac{\abs{x} + \sqrt{x^2-4}}{2}\right) \right) & \text{if } \abs{x} \geq 2. \end{cases}
\end{align*}

\begin{cor}
\label{cor:explicit}
For a pure $(p,q)$ model satisfying Assumption \ref{assn} with $\gamma = \frac{p}{p+q}$, we have
\[
    E_\infty\left(p,q,\frac{p}{p+q}\right) = 2\sqrt{\frac{p+q-1}{p+q}}, 
\]
and
\begin{align}
    \Sigma_{p+q}(t) \defeq \lim_{N \to \infty} \frac{1}{N}\log \E[\Crt_N^{\textup{tot}}(t)] &= \begin{cases} \frac{1+\log(p+q-1)}{2} + \Omega\left(t\sqrt{\frac{p+q}{p+q-1}}\right) - \frac{t^2}{2} & \text{if } t \leq 0, \\ \frac{\log(p+q-1)}{2} & \text{if } t \geq 0, \end{cases} \label{eqn:bsg_defsigmapq} \\
    \Sigma_{p+q} \defeq \lim_{N \to \infty} \frac{1}{N}\log \E[\Crt_N^{\textup{tot}}] &= \frac{\log(p+q-1)}{2} \notag, \\
    \Sigma_{p+q, \textup{min}}(t) \defeq \lim_{N \to \infty} \frac{1}{N}\log \E[\Crt_N^{\textup{min}}(t)] &= \Sigma_{p,q}\left(\min\left(t,-E_\infty\left(p,q,\frac{p}{p+q}\right)\right)\right) \notag \\
    \Sigma_{p+q,\textup{min}} \defeq \lim_{N \to \infty} \frac{1}{N}\log \E[\Crt_N^{\textup{min}}] &= \Sigma_{p,q}\left(-E_\infty\left(p,q,\frac{p}{p+q}\right)\right) = \frac{\log(p+q-1)}{2} + \frac{2}{p+q} - 1 \notag.
\end{align}
Notice the surprising fact that, as implicit in the notation, these functions depend only on $p+q$ rather than on $p$ and $q$ individually.
\end{cor}

The functions $\Sigma_{p+q}(t)$ are strictly increasing on $(-\infty,0)$, and $\Sigma_{p+q}(0) > 0$, so they each have a unique zero. They are plotted for $p+q = 4, 5, 6$ in Figure \ref{figure:sigmap's}. Notice that $p+q = 4$ (corresponding to a pure $(2,2)$ bipartite spin glass) is the smallest value to which Theorem \ref{thm:bsg} applies. As a corollary, we obtain a lower bound on the ground state of $\mc{H}_{N,p,q}$ in the classical way.

\begin{figure}
\begin{center}
\includegraphics[scale=0.6]{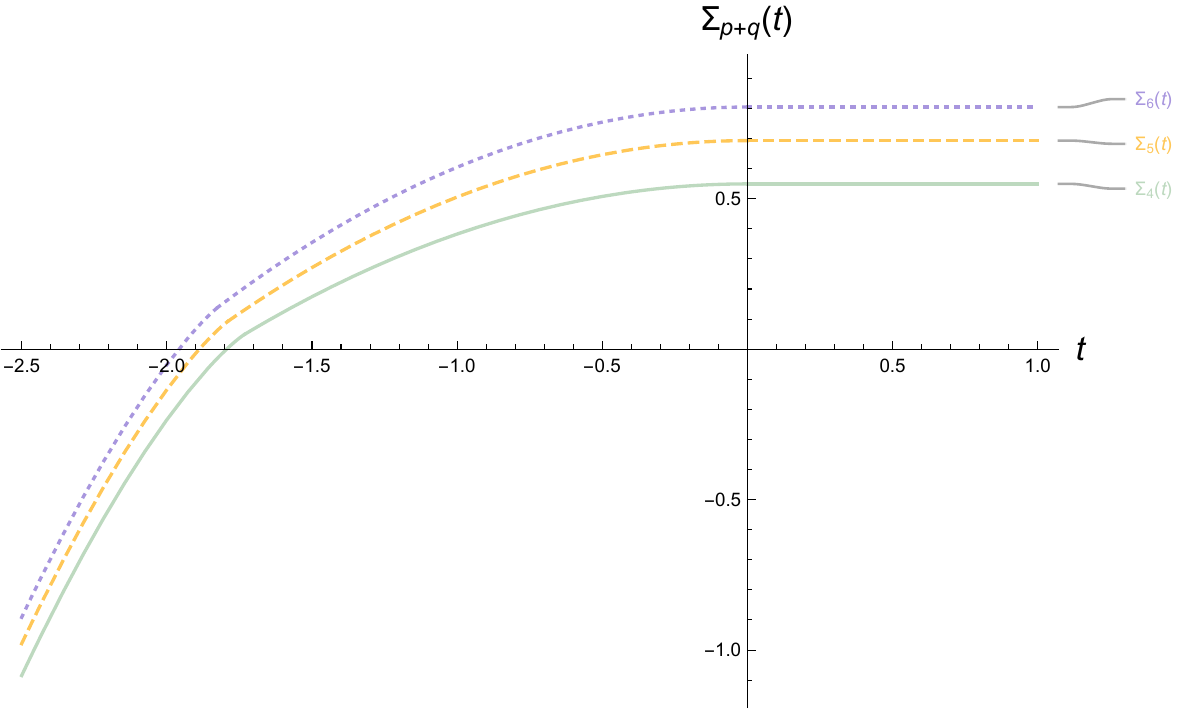}
\end{center}
\caption{Plots of $\Sigma_{p+q}(t)$, which captures the asymptotic complexity of total critical points with field values in $(-\infty,Nt)$ of the pure $(p,q)$ model at $\gamma = \frac{p}{p+q}$, for $p+q = 4, 5, 6$ (solid green, dashed yellow, dotted purple, respectively). Negative values of $\Sigma_{p+q}(t)$ are irrelevant for us, since we can prove that the zero of $\Sigma_{p+q}$ is a lower bound for the ground state (and we believe it is equal to the ground state). The functions stabilize at $t = 0$: this is consistent with distributional symmetry $\mc{H}_{N,p,q} \overset{d}{=} - \mc{H}_{N,p,q}$, since we would expect the total number of critical points to be twice the number of critical points with values in $(-\infty,0)$ on average.}
\label{figure:sigmap's}
\end{figure}

\begin{cor}
\label{cor:groundstate}
Let $-E_0(p+q)$ be the unique zero of the function $\Sigma_{p+q}$ defined in \eqref{eqn:bsg_defsigmapq}, and consider the Hamiltonian $\mc{H}_{N,p,q}$ of a pure $(p,q)$ model with $\gamma = \frac{p}{p+q}$. For any $\epsilon > 0$ there exist $C_1, C_2 > 0$ such that
\[
    \P\left(\min_{u,v} \mc{H}_{N,p,q}(u,v) \leq N(-E_0(p+q)-\epsilon)\right) \leq C_1\exp(-C_2N).
\]
Furthermore, 
\[
    \lim_{p+q \to \infty} \frac{E_0(p+q)}{\sqrt{\log(p+q)}} = 1.
\]
\end{cor}
\begin{rem}
One can compute numerically $-E_0(4) \approx -1.794$, $-E_0(5) \approx -1.888$, and $-E_0(6) \approx -1.959$.
\end{rem}

In fact, the functions $\Sigma_{p+q}(t)$ and $\Sigma_{p+q, \textup{min}}(t)$ have already appeared in the literature, in \cite{AufBenCer2013}: They give exactly the complexities of the numbers of critical points and of local minima, respectively, of a spherical pure $(p+q)$-spin glass below level $Nt$. That is, define the spherical pure $(p+q)$-spin Hamiltonian $\mc{H}_{N,p+q}$ over $\sigma = (\sigma_1, \ldots , \sigma_N) \in S^{N-1}$ by
\[
    \mc{H}_{N,p+q}(\sigma) = \frac{1}{N^{(p+q-1)/2}} \sum_{i_1, \ldots, i_{p+q}=1}^N J_{i_1, \ldots, i_{p+q}} \sigma_{i_1} \cdots \sigma_{i_{p+q}},
\]
where the $J$ variables are i.i.d. standard Gaussians, and let $\Crt_N^{\textup{pure } p+q}(t)$ be the number of critical points (and $\Crt_N^{\textup{pure } p+q, \textup{min}}(t)$ be the number of local minima) of $\mc{H}_{N,p+q}$ at which $\mc{H}_{N,p+q} \leq Nt$. Then \cite[Theorems 2.5, 2.8]{AufBenCer2013} show that
\[
    \lim_{N \to \infty} \frac{1}{N} \log \E[\Crt_N^{\textup{pure } p+q}(t)] = \Sigma_{p+q}(t), \quad \lim_{N \to \infty} \frac{1}{N} \log \E[\Crt_N^{\textup{pure } p+q, \textup{min}}(t)] = \Sigma_{p+q, \textup{min}}(t).
\]
(A computation shows that our $\Sigma_{p+q}$ and $\Sigma_{p+q, \textup{min}}$ are their $\Theta_{p+q}$ and $\Theta_{0, p+q}$, respectively. We have used their notation for $-E_0(p+q)$ in the same normalization.)

But we emphasize that, despite the superficial similarity between the pure $p+q$-spin Hamiltonian $\mc{H}_{N,p+q}$ and the pure bipartite $(p,q)$-spin Hamiltonian $\mc{H}_{N,p,q}$ with $\gamma = \frac{p}{p+q}$, they are different processes: Their covariance structures are (assuming $N_1 = \gamma N$ for clarity)
\begin{align*}
    \E[\mc{H}_{N,p+q}(\sigma)\mc{H}_{N,p+q}(\sigma')] &= N^{1-(p+q)} \left(\sum_{i=1}^N \sigma_i \sigma'_i \right)^{p+q}, \\
    \E[\mc{H}_{N,p,q}(u,v)\mc{H}_{N,p,q}(u',v')] &= N^{1-(p+q)} \frac{(p+q)^{p+q}}{p^pq^q} \left( \sum_{i=1}^{\gamma N} u_iu'_i \right)^p \left( \sum_{i=1}^{(1-\gamma)N} v_iv'_i \right)^q.
\end{align*}

\begin{rem}
It seems likely that Assumption \ref{assn}, i.e. the restriction ``neither a linear combination of $(1,q)$ spins for different $q$ values, nor a linear combination of $(p,1)$ spins for different $p$ values,'' is technical rather than fundamental. In \cite{AufChe2014}, the restriction is ``not a pure $(1,1)$ spin.'' See Remark \ref{rem:pure1q} for a discussion of the obstacles.
\end{rem}


\section{Proofs}
\label{sec:bsg}

\subsection{Discussion of proof techniques for the toy model of GOE.}\
\label{subsec:goe} 

The goal of this subsection is to explain our proof techniques, specifically the combination of multiple local laws with respect to different Dyson equations. We do so by sketching a proof of the following (obviously classical) result, which is formally unconnected to the main results of this paper, but rather just a simple setting in which to fix ideas.

Recall that an $N \times N$ GOE matrix  $H_N$ is a real symmetric matrix, with independent entries up to symmetry, which are centered Gaussian variables with $\E[(H_N)_{ij}^2] = \frac{1+\delta_{ij}}{N}$.

\begin{lem}
For each $z \in \mathbb{H}$, there exists a unique solution $m_\infty(z) \in \C$ to the problem
\begin{equation}
\label{eqn:mdegoe}
	1+(z+m_\infty(z))m_\infty(z) = 0, \quad \im m(z) > 0,
\end{equation}
and $m_\infty(z)$ is the Stieltjes transform of a probability measure $\mu_\infty$ on $\R$ with compact support and a bounded, H\"olderian density $\mu_\infty(\cdot)$ with respect to Lebesgue measure.

Furthermore, if $H_N$ is an $N \times N$ GOE matrix, then
\begin{equation}
\label{eqn:goe_example_distance}
	d_{\textup{BL}}(\E[\hat{\mu}_{H_N}],\mu_\infty) \leq N^{-\epsilon}
\end{equation}
for some $\epsilon > 0$, and the largest (resp. smallest) eigenvalue of $H_N$ tends in probability to the right endpoint $r(\mu_\infty)$ (resp. the left endpoint $\ell(\mu_\infty)$).
\end{lem}

Of course, the measure $\mu_\infty$ in the above is the semicircle law, $m_\infty(z)$ has the closed form $m_\infty(z) = \frac{-z+\sqrt{z^2-4}}{2}$, and there are much simpler proofs of this result than the one we are about to sketch. The point is to see what we can do without knowing these closed forms, since we will not have closed forms for bipartite spin glasses. That is, we want to give a global law for $H_N$ with respect to a measure $\mu_\infty$ which has (a) as simple a description as possible (meaning \eqref{eqn:mdegoe} is over scalars, not matrices) and (b) many nice properties. We also give ourselves an additional handicap: We do not allow for ``free'' the result that the Dyson equation \eqref{eqn:introvonneumannmde} over $\C$ has a unique solution which is the Stieltjes transform for some measure. We do this since, in the bipartite case, this result will \emph{not} be for free; there the problem is over $\C^2$ instead of $\C$, and the operator $\mc{S}$ in \eqref{eqn:introvonneumannmde} is \emph{not} symmetric, so there is a unique solution, but it takes work to show that it is the Stieltjes transform of some measure. 

The following definition first appeared in \cite{AjaErdKru2019}.

\begin{defn}
\label{defn:flatness}
A sequence of operators $(\mc{S}_N : \C^{N \times N} \to \C^{N \times N})_{N=1}^\infty$ is called \emph{flat} if it preserves the cone of positive semidefinite matrices and there exists $\kappa > 0$ such that, for all $N$ and all positive semidefinite $N \times N$ matrices $T$, one has
\[
	\frac{1}{\kappa N} \Tr(T) \leq \mc{S}_N[T] \leq \frac{\kappa}{N}\Tr(T).
\]
(Here we use the shorthand notation $a \leq M \leq b$, where $a, b > 0$ and $M$ is a positive semidefinite matrix, for the inequality $a\Id \leq M \leq b\Id$ in the sense of quadratic forms.)
\end{defn}

\begin{proof}[Proof sketch.]
For each $N$, consider the Dyson equation \eqref{eqn:introvonneumannmde} over $\ms{A} = \C^{N \times N}$ with $A = 0$ and with the positivity-preserving linear operator $\mc{S}'_N : \C^{N \times N} \to \C^{N \times N}$ given by $\mc{S}'_N[T] = \frac{1}{N}\Tr(T)\Id_{N \times N}$, which is symmetric with respect to the Frobenius inner product $\ip{T,U} = \Tr{T^\ast U}$. We claim that solutions to this equation are in bijection with solutions to \eqref{eqn:mdegoe} (eventually, of course, there will only be one solution on each side). Indeed, since $\mc{S}'_N$ maps into matrices of the form constant times identity, one sees directly from the Dyson equation that the unique solution matrix must have the form $M_N(z) = m_\infty(z) \Id_{N \times N}$ for some scalar $m_\infty(z) \in \mathbb{H}$ that does not depend on $N$ and that solves \eqref{eqn:mdegoe}. In the other direction, if $m_\infty(z)$ solves \eqref{eqn:mdegoe}, then $m_\infty(z)\Id_{N \times N}$ solves the matrix problem. Thus existence, uniqueness, and the property ``solution is the Stieltjes transform of some compactly supported measure $\mu_\infty$'' descend to \eqref{eqn:mdegoe} from the corresponding results at the matrix level (\cite{HelFarSpe2007} and \cite[Proposition 2.1]{AltErdKru2020}).

One can check that the operator $\mc{S}'_N$ is flat; thus \cite[Proposition 2.2]{AjaErdKru2019} gives the existence of a H\"olderian density $\mu_\infty(\cdot)$ for $\mu_\infty$.

The results of \cite{AltErdKruNem2019}, combined with standard arguments we will make elsewhere in this paper, give 
\[
	d_{\textup{BL}}(\E[\hat{\mu}_{H_N}],\mu_N) \leq N^{-\epsilon}
\]
for some $\epsilon > 0$, where $\mu_N$ is found by solving the Dyson equation \eqref{eqn:introvonneumannmde} over $\ms{A} = \C^N$ with $A = 0$, multiplication in $\C^N$ taken componentwise (i.e., $(a_1, \ldots, a_N)(b_1, \ldots, b_N) = (a_1b_1, \ldots, a_Nb_N)$), and self-energy operator
\[
	\ms{S}_N[(r_1,\ldots,r_N)] = \frac{\sum_{k=1}^N r_k}{N} (1, \ldots, 1) + \frac{1}{N}(r_1,\ldots,r_N).
\]
To show $d_{\textup{BL}}(\mu_N,\mu_\infty) \leq N^{-\epsilon}$, we claim that solutions to this equation are in bijection to solutions of \eqref{eqn:introvonneumannmde} over $\ms{A} = \C^{N \times N}$ with $A = 0$ and self-energy operator $\mc{S}_N[T] = \frac{1}{N} \Tr(T) \Id_{N \times N} + \frac{1}{N}\diag(T)$. Indeed, if $(m_1(z),\ldots,m_N(z))$ solves the problem over $\C^N$, then $\diag(m_1(z),\ldots,m_N(z))$ solves the problem over $\C^{N \times N}$. On the other hand, since $\mc{S}_N$ maps into diagonal matrices, solutions to the problem over $\C^{N \times N}$ must also be diagonal. If it happens that $\diag(m_1(z),\ldots,m_N(z))$ is such a solution, then $(m_1(z),\ldots,m_N(z))$ solves the problem over $\C^N$. Thus these two problems induce the \emph{same} sequence $(\mu_N)_{N=1}^\infty$ of probability measures on $\R$.

But $\mc{S}_N$ and $\mc{S}'_N$ from above, both operators on $\C^{N \times N}$, are in some sense ``close,'' so the measures they induce (namely $\mu_N$ and $\mu_\infty$, respectively) should be close as well. Lemmas useful for formalizing this were proved in \cite[Proposition 3.1]{BenBouMcK2022}, \cite[Lemma 3.1]{BenBouMcK2021II}, and can be used to finish the proof of \eqref{eqn:goe_example_distance}. In essentially the same way, one can give a lower bound on the largest eigenvalue of $H_N$: if it stays below $r(\mu_\infty)-\epsilon$, \eqref{eqn:goe_example_distance} is violated, as witnessed by a test function supported on $(r(\mu_\infty)-\epsilon,r(\mu_\infty))$. It remains to show that this top eigenvalue does not become an outlier.

To show this, we study another random matrix coupled with $H_N$, namely the random matrix $H'_N$ defined entrywise as
\[
	(H'_N)_{ij} = \begin{cases} (H_N)_{ij} & \text{if } i \neq j, \\ \frac{1}{\sqrt{2}} (H_N)_{ii} & \text{if } i = j. \end{cases}
\]
That is, $H'_N$ is like a GOE matrix but with the normalization $\E[(H'_N)_{ij}^2] = \frac{1}{N}$ instead of $\E[(H_N)_{ij}^2] = \frac{1+\delta_{ij}}{N}$. On the one hand, we can apply \cite[Theorem 2.4]{AltErdKruNem2019} to $H'_N$: For some measure $\mu'_N$ arising from the relevant Dyson equation, this result shows $\P(\lambda_{\textup{max}}(H'_N) > r(\mu'_N) + \epsilon) \leq C_\epsilon N^{-100}$ for every $\epsilon > 0$. But actually $\mu'_N = \mu_\infty$: indeed, the relevant Dyson equation from \cite{AltErdKruNem2019} is over $\ms{A} = \C^N$, with $A = 0$ and the self-energy operator $\ms{S}'_N[(r_1,\ldots,r_N)] = \frac{\sum_{k=1}^N r_k}{N} (1,\ldots,1)$. In the same way as above, one can show that solutions to this equation are in bijection with solutions to the problem we first considered, i.e., the problem over $\C^{N \times N}$ with operator $\mc{S}'_N$, which we showed induced the measure $\mu_\infty$. On the other hand, the Weyl inequalities give $\lambda_{\textup{max}}(H_N) \leq \lambda_{\textup{max}}(H'_N) + \|H_N - H'_N\|$. The matrix $H_N - H'_N$ is diagonal, and its diagonal entries are independent Gaussians with variance order $1/N$; thus $\P(\|H_N-H'_N\| > \epsilon) \lesssim \exp(-c_\epsilon N)$ for some constant $c_\epsilon > 0$, which completes the sketch of the proof.
\end{proof}

A note on terminology: It is common to call a Dyson equation over $\ms{A} = \C^{N \times N}$ the ``matrix Dyson equation.'' We will sometimes also call Dyson equations over $\C^N$ the ``matrix Dyson equation.'' The reason is that, as in the proof above, we think of them as morally defined on diagonal $N \times N$ matrices. When necessary, we distinguish between the two by talking about an ``MDE over matrices'' or an ``MDE over scalars.''

\subsection{Preliminaries on matrices, measures, and the Dyson equation.}\

In this section, we introduce the random matrices $H_N(u)$ which will appear in the Kac--Rice formula, along with related random matrices $H'_N(u)$ which are easier to work with, and corresponding measures $\mu_N(u)$ and $\mu_\infty(u)$ from the Dyson equation. Then we prove that these matrices satisfy a number of nice properties, both on their own and with respect to the measures $\mu_N(u)$ and $\mu_\infty(u)$. The reason is that the proofs of our main theorems in this paper are essentially of the form ``Apply \cite[Theorem 4.1, Theorem 4.5]{BenBouMcK2022} to the matrices $H_N(u)$''; the goal of this section is to check the conditions of those theorems. We will also prove Lemma \ref{lem:bsg_scalars}.

\begin{notn}
We write
\[
    I_1 = \llbracket 1, N_1-1 \rrbracket, \qquad I_2 = \llbracket N_1, N-2 \rrbracket.
\]
\end{notn}

For each $u \in \R^3$, define $A_N(u), A'_N(u) \in \R^{((N_1-1)+(N_2-1)) \times ((N_1-1)+(N_2-1))}$ by
\begin{align*}
    A_N(u) &= A_N(u_0,u_1,u_2) = \begin{pmatrix} \frac{N}{N_1}(\frac{\alpha_1\sigma_+}{2}u_1 + \frac{\alpha_1\sigma_-}{2}u_2 - \xi'_1 u_0) \Id_{N_1-1} & 0 \\ 0 & \frac{N}{N_2}(\frac{\alpha_2\sigma_-}{2}u_1 + \frac{\alpha_2\sigma_+}{2}u_2 - \xi'_2u_0) \Id_{N_2-1} \end{pmatrix}, \\
    A'_N(u) &= A'_N(u_0,u_1,u_2) = \begin{pmatrix} \frac{1}{\gamma}(\frac{\alpha_1\sigma_+}{2}u_1 + \frac{\alpha_1\sigma_-}{2}u_2 - \xi'_1 u_0) \Id_{N_1-1} & 0 \\ 0 & \frac{1}{1-\gamma}(\frac{\alpha_2\sigma_-}{2}u_1 + \frac{\alpha_2\sigma_+}{2}u_2 - \xi'_2u_0) \Id_{N_2-1} \end{pmatrix}.
\end{align*}
Next, we define random matrices $W_N, W'_N \in \R^{((N_1-1) + (N_2-1)) \times ((N_1-1) + (N_2-1))}$ one block at a time. Write $G$ for an $(N_1-1) \times (N_2-1)$ matrix with i.i.d. centered Gaussian entries, each with variance $\frac{N\xi''_{12}}{N_1N_2}$. For each $i = 1, 2$, let $G_i = \sqrt{\frac{N(N_i-1)\xi''_i}{N_i^2}}M^{N_i}$, where each $M^{N_i}$ is an $(N_i-1) \times (N_i-1)$ GOE matrix with normalization $\E[(M^{N_i})_{ij}^2] = \frac{1+\delta_{ij}}{N_i-1}$ (notice this normalization differs from that of \cite{AufChe2014}), and where the $M^{N_i}$'s are independent of each other and of $G$. Then we define $W_N$ by
\[
    W_N = \begin{pmatrix} G_1 & G \\ G^T & G_2 \end{pmatrix}.
\]
Let $T_N \in \R^{((N_1-1) + (N_2-1)) \times ((N_1-1) + (N_2-1))}$ be given entrywise by
\[
    (T_N)_{jk} = \begin{cases} \sqrt{\frac{N_1^2}{(1+\delta_{jk})\gamma^2N(N-2)}} & \text{if } j, k \in I_1, \\
    \sqrt{\frac{N_1N_2}{\gamma N(N_2-1)}} = \sqrt{\frac{N_1N_2}{(1-\gamma) N(N_1-1)}} & \text{if } j \in I_1, k \in I_2 \text{ or } j \in I_2, k \in I_1, \\ 
    \sqrt{\frac{N_2^2}{(1+\delta_{jk})(1-\gamma)^2N(N-2)}} & \text{if } j, k \in I_2, \end{cases}
\]
and let
\[
    W'_N = T_N \odot W_N.
\]
(That is, $W'_N$ is like $W_N$, but all the variances are multiplied by a carefully chosen factor which is close to one off the diagonal and close to $1/2$ on the diagonal.) Finally, let
\[
    H_N(u) = A_N(u) + W_N, \qquad H'_N(u) = A'_N(u) + W'_N.
\]
The matrix $H_N(u)$ is the one naturally appearing in the Kac--Rice formula, as we shall see, but it is well approximated by the matrix $H'_N(u)$, which is easier to work with, since the sequence of \emph{its} Dyson-equation measures is constant.

While the definitions are fresh, we store the following lemma for later use:

\begin{lem}
\label{lem:bsg_hn_vs_hprimen}
For every $R > 0$ and every $\epsilon > 0$, we have
\[
    \sup_{u \in B_R(0)} \P(\|H_N(u) - H'_N(u)\| \geq \epsilon) = \OO_{R,\epsilon}(e^{-N^{0.49}}).
\]
\end{lem}
\begin{proof}
Write $E_N = W_N - W'_N$. From the definitions, we check that $E_N$ is a matrix of independent Gaussian entries up to symmetry, and that there exists some constant $C = C_{\xi,\gamma}$ such that the off-diagonal entries of $E_N$ have variance at most $C/N^3$ and the diagonal entries have variance at most $C/N$. If $\|\cdot\|_{\max{}}$ is the maximum norm for matrices, we thus have
\[
    \P\left(\|E_N - \diag(E_N)\|_{\max{}} \geq \frac{C}{N^{5/4}}\right) \leq \frac{N(N-1)}{2}\P\left(\abs{\mc{N}(0,1)} \geq N^{1/4}\right) \leq N^2e^{-\frac{\sqrt{N}}{2}}.
\]
Then 
\begin{align*}
    \P(\|E_N\| \geq \epsilon) &\leq \P\left(\|E_N\| \geq \epsilon, \|E_N - \diag(E_N)\| \leq \frac{C}{N^{1/4}}\right) + N^2e^{-\frac{\sqrt{N}}{2}} \leq \P\left(\|\diag(E_N)\| \geq \frac{\epsilon}{2}\right) + N^2e^{-\frac{\sqrt{N}}{2}},
\end{align*}
where the last inequality holds for $N$ large enough. But now $\diag(E_N)$ has independent Gaussian entries with variance order $1/N$, so $\P(\|\diag(E_N)\| \geq \epsilon/2)$ is order $e^{-cN}$, up to polynomial factors in $N$, for some $c = c_{\xi,\gamma,\epsilon}$; thus
\[
    \P(\|E_N\| \geq \epsilon) = \OO(e^{-N^{0.49}}),
\]
say. On the other hand, we have
\begin{equation}
\label{eqn:lemma3.1.3.6}
\begin{split}
    \|A_N(u) - A'_N(u)\| &= \max\left\{ \abs{\frac{N}{N_1} - \frac{1}{\gamma}}\abs{\frac{\alpha_1\sigma_+}{2}u_1 + \frac{\alpha_1\sigma_-}{2} u_2 - \xi'_1u_0}, \abs{\frac{N}{N_2} - \frac{1}{1-\gamma}} \abs{\frac{\alpha_2\sigma_-}{2}u_1 + \frac{\alpha_2\sigma_+}{2}u_2 - \xi'_2u_0} \right\} \\
    &= \OO\left(\frac{\|u\|}{N}\right)
\end{split}
\end{equation}
Since
\[
    \P(\|H_N(u) - H'_N(u)\| \geq \epsilon) \leq \P\left(\|W_N - W'_N\| \geq \frac{\epsilon}{2}\right) + \mathbf{1}_{\|A_N(u) - A'_N(u)\| \geq \frac{\epsilon}{2}},
\]
this completes the proof.
\end{proof}

In the following, we use the sequences
\begin{alignat*}{3}
	&b_{11}^{(N)} = \frac{N\xi''_1}{N_1^2}, \qquad &&b_{22}^{(N)} = \frac{N\xi''_2}{N_2^2}, \qquad & &b_{12}^{(N)} = \frac{N\xi''_{12}}{N_1N_2}, \\
	&\widetilde{b_{11}}^{(N)} = \frac{\xi''_1}{\gamma(N_1-1)}, \qquad &&\widetilde{b_{22}}^{(N)} = \frac{\xi''_2}{(1-\gamma)(N_2-1)}, \qquad & &\widetilde{b_{12}}^{(N)} = \frac{\xi''_{12}}{\gamma(N_2-1)} = \frac{\xi''_{12}}{(1-\gamma)(N_1-1)}, 
\end{alignat*}
where the last equality follows from \eqref{eqn:ratio_gamma}, as well as the matrix
\[
    \tau = \diag\left( \underbrace{b_{11}^{(N)}, b_{11}^{(N)}, \ldots}_{N_1-1 \text{ times}}, \underbrace{b_{22}^{(N)}, b_{22}^{(N)}, \ldots}_{N_2-1 \text{ times}} \right),
\]
and the linear operators $\mc{S}_N, \mc{S}'_N : \C^{(N-2) \times (N-2)} \to \C^{(N-2) \times (N-2)}$ defined by their action on block matrices $T = \left(\begin{smallmatrix} T_{11} & T_{12} \\ T_{21} & T_{22} \end{smallmatrix}\right)$ (with the sizes $T_{ij} \in \C^{(N_i-1) \times (N_j-1)}$ for $i, j= 1,2$) by
\begin{align}
\label{eqn:bsg_defofsN}
\begin{split}
    \mc{S}_N\left[ \begin{pmatrix} T_{11} & T_{12} \\ T_{21} & T_{22} \end{pmatrix} \right] &= \begin{pmatrix} \left(b_{11}^{(N)} \Tr(T_{11}) + b_{12}^{(N)}\Tr(T_{22})\right)\Id & 0 \\ 0 & \left(b_{12}^{(N)} \Tr(T_{11}) + b_{22}^{(N)}\Tr(T_{22})\right) \Id \end{pmatrix} + \tau \odot \diag(T), \\
    \mc{S}'_N\left[ \begin{pmatrix} T_{11} & T_{12} \\ T_{21} & T_{22} \end{pmatrix} \right] &= \begin{pmatrix} \left(\widetilde{b_{11}}^{(N)} \Tr(T_{11}) + \widetilde{b_{12}}^{(N)} \Tr(T_{22})\right) \Id & 0 \\ 0 & \left(\widetilde{b_{12}}^{(N)} \Tr(T_{11}) + \widetilde{b_{22}}^{(N)} \Tr(T_{22})\right) \Id \end{pmatrix}.
\end{split}
\end{align}
(Recall that $\odot$ is the entrywise (Hadamard) product of matrices and $\diag(T)$ is the diagonal matrix obtained from $T$ by setting all off-diagonal entries to zero, so $\mc{S}_N$ and $\mc{S}'_N$ map into diagonal matrices.)

\begin{lem}
\label{lem:bsg_flatness}
Under Assumption \ref{assn}, the sequences $(\mc{S}_N)_{N=1}^\infty$ and $(\mc{S}'_N)_{N=1}^\infty$ are flat in the sense of Definition \ref{defn:flatness}. Furthermore, we have
\begin{align}
	\sup_N \max(\|A_N(u)\|,\|A'_N(u)\|) &= \OO(\|u\|), \label{eqn:bsg_3.1.3.3} \\
	\sup_N \max(\|\mc{S}_N\|, \|\mc{S}'_N\|, \|\mc{S}_N\|_{\textup{hs} \to \|\cdot\|}, \|\mc{S}'_N\|_{\textup{hs} \to \|\cdot\|}) &< \infty, \label{eqn:bsg_sNopnorm} \\ 
	\|\mc{S}_N - \mc{S}'_N\| &= \OO\left(\frac{1}{N}\right). \label{eqn:bsg_snclosesn'} \\
	\|A_N(u) - A'_N(u)\| &= \OO\left(\frac{\|u\|}{N}\right) \label{eqn:bsg_3.1.3.6}
\end{align}
\end{lem}
\begin{proof}
Since $\abs{\frac{N_1}{N} - \gamma} = \OO(\frac{1}{N})$ and $\xi''_1, \xi''_2 > 0$, we can find $\kappa$ such that
\[
    \frac{1}{\kappa(N-2)} \leq b_{11}^{(N)}, b_{22}^{(N)}, b_{12}^{(N)}, \widetilde{b_{11}}^{(N)}, \widetilde{b_{22}}^{(N)}, \widetilde{b_{12}}^{(N)} \leq \frac{\kappa}{N-2}.
\]
If $T \geq 0$, then $0 \leq \tau \odot \diag(T) \leq \frac{\kappa}{N-2}\Tr(T)$; this suffices to prove flatness for both operators. Since $A_N$ and $A'_N$ are diagonal with bounded entries on the diagonal (via \eqref{eqn:ratio_gamma}), \eqref{eqn:bsg_3.1.3.3} is immediate. The estimates
\begin{equation}
\label{eqn:bsg_hadamard}
    \|\tau \odot \diag(T)\| \leq \frac{\kappa}{N-2}\|T\|
\end{equation}
and
\begin{align*}
    \abs{b_{\delta}^{(N)} \Tr(T_{11}) + b_{\gamma}^{(N)} \Tr(T_{22})} \leq \frac{\kappa}{N-2}(\abs{\Tr(T_{11})} + \abs{\Tr(T_{22})}) \leq \kappa\|T\| \\
    \abs{\widetilde{b_{\delta}}^{(N)} \Tr(T_{11}) + \widetilde{b_{\gamma}^{(N)}} \Tr(T_{22})} \leq \frac{\kappa}{N-2}(\abs{\Tr(T_{11})} + \abs{\Tr(T_{22})}) \leq \kappa\|T\|
\end{align*}
(valid for any $\delta, \gamma \in \{11, 12, 22\}$) establish \eqref{eqn:bsg_sNopnorm}. Finally, if we define for $\delta \in \{11, 12, 22\}$ the sequences 
\[
	a_{\delta}^{(N)} = \left(b_{\delta}^{(N)} - \widetilde{b_{\delta}}^{(N)}\right)(N-2),
\]
then using \eqref{eqn:bsg_hadamard} we conclude 
\[
	\|\mc{S}_N(T) - \mc{S}'_N(T)\| \leq \left( \frac{\kappa}{N-2} + \max\{ |a_{11}^{(N)}| + |a_{12}^{(N)}|, |a_{12}^{(N)}| + |a_{22}^{(N)}|\}\right) \|T\|
\]
for all $T$. But we assumed $\frac{N_1-1}{N-2} = \gamma$ in \eqref{eqn:ratio_gamma}, which tells us $\max(|a_{11}^{(N)}|, |a_{12}^{(N)}|, |a_{22}^{(N)}|) = \OO(1/N)$; this completes the proof of \eqref{eqn:bsg_snclosesn'}. The estimate \eqref{eqn:bsg_3.1.3.6} was actually established above, in \eqref{eqn:lemma3.1.3.6}; we only copy it here since we want to cite \eqref{eqn:bsg_3.1.3.3} - \eqref{eqn:bsg_3.1.3.6} at once below.
\end{proof}

\begin{lem}
\label{lem:mde_properties}
Fix any $u \in \R^3$. Consider the MDEs over $\C^{(N-2) \times (N-2)}$ 
\begin{align}
	\Id + (z\Id - A_N(u) + \mc{S}_N[M_N(u,z)])M_N(u,z) = 0 \quad \text{subject to} \quad \im M_N(u,z) > 0, \label{eqn:mdesn} \\
	\Id + (z\Id - A'_N(u) + \mc{S}'_N[M'_N(u,z)])M'_N(u,z) = 0 \quad \text{subject to} \quad \im M'_N(u,z) > 0. \label{eqn:mdes'n}
\end{align}
Then $\frac{1}{N-2}\Tr(M_N(u,z))$ is the Stieltjes transform of a probability measure $\mu_N(u)$ on $\R$, and $\frac{1}{N-2}\Tr(M'_N(u,z))$ is the Stieltjes transform of a probability measure $\mu_\infty(u)$ on $\R$ which does not depend on $N$. These measures have densities $\mu_N(u,\cdot)$ and $\mu_\infty(u,\cdot)$, which are compactly supported, H\"olderian, and bounded locally uniformly in $u$, meaning there exists $C(u)$ with $\sup_{u \in B_R(0)} C(u) < \infty$ for every $R > 0$ and 
\begin{align*}
	(\supp \mu_\infty(u)) \cup (\cup_{N=1}^\infty \supp \mu_N(u)) &\subset [-C(u),C(u)], \\
	\sup_{x \neq y} \frac{\max(\max_N \abs{\mu_N(u,x)-\mu_N(u,y)}, \abs{\mu_\infty(u,x) - \mu_\infty(u,y)})}{\abs{x-y}^c} &\leq C(u), \\
	\|\mu_N(u,\cdot)\|_{L^\infty} &\leq C(u)
\end{align*}
for some small, universal $c > 0$.
\end{lem}
\begin{proof}
Both $\mc{S}_N$ and $\mc{S}'_N$ are linear operators, preserving the cone of positive semidefinite matrices, that are symmetric with respect to the Frobenius inner product. Thus the normalized traces of $M_N(u,z)$ and $M'_N(u,z)$ are the Stieltjes transforms of some compactly-supported probability measures $\mu_N$ and $\mu'_N$ on $\R$ (later we will show $\mu'_N$ does not depend on $N$), and if we define
\begin{align}
	\kappa(u) &= \sup_N\|A_N(u)\| + 2(\sup_N\|\mc{S}_N\|)^{1/2}, \label{eqn:munsupport} \\
	\kappa'(u) &= \sup_N\|A'_N(u)\| + 2(\sup_N\|\mc{S}'_N\|)^{1/2}, \notag
\end{align}
then \cite[Proposition 2.1]{AjaErdKru2019} (see (2.7), (2.8) there) gives the support conditions $\cup_{N=1}^\infty \supp(\mu_N(u)) \subset [-\kappa(u),\kappa(u)]$ and $\cup_{N=1}^\infty \supp(\mu'_N(u)) \subset [-\kappa'(u),\kappa'(u)]$. From Lemma \ref{lem:bsg_flatness}, we have $\sup_{u \in B_R(0)} \max(\kappa(u),\kappa'(u)) < \infty$. Since $\mc{S}_N$ is flat, \cite[Proposition 2.2]{AjaErdKru2019} yields (a) that each $\mu_N(u)$ admits a density $\mu_N(u,\cdot)$ with respect to Lebesgue measure, and (b) that each $\mu_N(u,\cdot)$ is H\"olderian, with a H\"older exponent that is universal and a H\"older constant that can be taken uniformly over $u \in B_R(0)$. Hence the densities are bounded, uniformly over $u \in B_R(0)$. Since $\mc{S}'_N$ is flat, the same holds for $\mu'_N(u)$.

It remains only to show that each sequence $\mu'_N(u)$ is in fact constant. Since $\mc{S}'_N$ maps into diagonal matrices and $A'_N(u)$ is diagonal, we can see directly from the MDE \eqref{eqn:mdes'n} that $M'_N(u,z)$ must be diagonal. By looking at the MDE componentwise, we see that the entries on the diagonal can only take two values, which we will call $m_1(u,z)$ (for the first $N_1-1$ entries) and $m_2(u,z)$ (for the last $N_2-1$ entries). Writing \eqref{eqn:mdes'n} out in components shows that $\{m_1(u,z),m_2(u,z)\}$ is a solution to \eqref{eqn:bsg_scalars}, i.e., $m_1(u,z)$ and $m_2(u,z)$ do not depend on $N$. Then from \eqref{eqn:ratio_gamma} we have that
\[
	\frac{1}{N-2}\Tr(M'_N(u,z)) = \frac{N_1-1}{N-2} m_1(u,z) + \frac{N_2-1}{N-2} m_2(u,z) = \gamma m_1(u,z) + (1-\gamma) m_2(u,z)
\]
does not depend on $N$ either, so $\mu'_N$ is independent of $N$ as claimed.
\end{proof}

\begin{proof}[Proof of Lemma \ref{lem:bsg_scalars}]
In the proof of Lemma \ref{lem:mde_properties}, we showed that \eqref{eqn:bsg_scalars} has at least one solution, and that, for this solution, $\gamma m_1(u,z) + (1-\gamma) m_2(u,z)$ is the Stieltjes transform of a probability measure $\mu_\infty(u)$ with the desired properties, so all that remains is uniqueness. But uniqueness for \eqref{eqn:bsg_scalars} follows from uniqueness for \eqref{eqn:mdes'n}, since one can check that
\[
	M'_N(u,z) = \diag\left( \underbrace{m_1(u,z), \ldots, m_1(u,z)}_{N_1-1 \text{ times}}, \underbrace{m_2(u,z), \ldots, m_2(u,z)}_{N_2 - 1 \text{ times}} \right)
\]
exhibits a solution to \eqref{eqn:mdes'n} whenever $\{m_1(u,z),m_2(u,z)\}$ solves \eqref{eqn:bsg_scalars}.
\end{proof}

\begin{lem}
\label{lem:bsg_checkingassnsK}

For each $u \in \R^3$, the sequence $(H_N(u))_{N=1}^\infty$ of random matrices satisfies the assumptions of \cite[Theorem 1.2]{BenBouMcK2022}. In fact, the assumptions are satisfied uniformly over compact sets of $u$ (``locally uniformly in $u$'').\footnote{This means, for example, that \cite[(1-10)]{BenBouMcK2022}, which reads ``there exists $\kappa > 0$ such that ${\rm W}_1(\E[\hat{\mu}_{H_N}],\mu_N) \leq N^{-\kappa}$,'' is replaced with ``for every compact $K \subset \R^3$, there exists $\kappa = \kappa(K) > 0$ with $\sup_{u \in K} {\rm W}_1(\E[\hat{\mu}_{H_N}(u)],\mu_N(u)) \leq N^{-\kappa}$,'' and similarly for the other assumptions.}
Additionally,
\begin{equation}
\label{eqn:bsg_wnleftedge}
    \liminf_{N \to \infty} \lambda_{\min{}}(W_N) \geq -2\sqrt{\sup_N \|\mc{S}_N\|} - 1 \quad \text{a.s.}
\end{equation}
\end{lem}
\begin{proof}
The structure of \cite{BenBouMcK2022} is as follows: Theorems 1.1 and 1.2 there are general theorems saying ``if random matrices $H_N$ and measures $\mu_N$ satisfy certain assumptions, then determinant concentration holds, in the sense that $\lim_{N \to \infty} \left(\frac{1}{N}\log \E[\abs{\det(H_N)}] - \int_\R \log\abs{\lambda}\mu_N(\diff \lambda)\right) = 0$.'' They are followed by a sequence of corollaries of the form ``random matrices in this family (for example, Wigner matrices, Erd\H{o}s-R\'enyi matrices, etc.), with certain measures $\mu_N$, satisfy determinant concentration.'' The proof of each of these corollaries looks like either ``random matrices in this family, with these measures $\mu_N$, satisfy the assumptions of Theorem 1.1'' or ``random matrices in this family, with these measures $\mu_N$, satisfy the assumptions of Theorem 1.2.''

Our random matrices $H_N(u)$ actually fit into two of the families considered in these corollaries (called ``Gaussian matrices with a (co)variance profile'' and ``Block-diagonal Gaussian matrices''), both of which are proved via Theorem 1.2. It is technically more convenient to consider the latter (here with just one block). That is, we only need to show that the matrices $H_N(u)$ satisfy the assumptions of Corollary 1.10 in \cite{BenBouMcK2022} with $K = 1$, which are called (MS), (MF), and (R); then the proof of this corollary there shows that the matrices also satisfy the assumptions of \cite[Theorem 1.2]{BenBouMcK2022}. In fact, a close examination of the proof of Corollary 1.10 shows that, to show the assumptions of Theorem 1.2 locally uniformly in $u$, it suffices to show the assumptions of Corollary 1.10 locally uniformly in $u$.

The bounded-mean condition (MS) was verified, locally uniformly in $u$, in \eqref{eqn:bsg_3.1.3.3}. The mean-field-randomness condition (MF), which says roughly that the variances $s_{jk} = \Var((W_N)_{jk})$ are all order $1/N$, is clear, since
\[
    s_{jk} = \begin{cases} \frac{N\xi''_1(1+\delta_{jk})}{N_1^2} & \text{if } j, k \in I_1, \\ 
    \frac{N\xi''_{12}}{N_1N_2} & \text{if } j \in I_1, k \in I_2 \text{ or } j \in I_2, k \in I_1, \\ 
    \frac{N\xi''_2(1+\delta_{jk})}{N_2^2} & \text{if } j, k \in I_2. \end{cases}
\]
Notice also that this condition does not depend on $u$ at all, since $W_N$ does not depend on $u$.

Now we check the regularity assumption (R). The appropriate Dyson equation in this context \cite[(1-19)]{BenBouMcK2022} is a system of $N-2$ scalar equations, which for our matrices takes the following form: Define the operators $\ms{S}_i : \C^{N-2} \to \C$ by
\[
    \ms{S}_i[\mathbf{r}] = \begin{cases} \frac{N\xi''_1}{N_1^2} \sum_{k \in I_1} (1+\delta_{ik}) r_k + \frac{N\xi''_{12}}{N_1N_2} \sum_{k \in I_2} r_k & \text{if } i \in I_1, \\ \frac{N\xi''_{12}}{N_1N_2} \sum_{k \in I_1} r_k + \frac{N\xi''_2}{N_2^2} \sum_{k \in I_2} (1+\delta_{ik})r_k & \text{if } i \in I_2. \end{cases}
\]
Then seek the unique solution $\mathbf{m}(u,z) = (m_1(u,z), \ldots, m_{N-2}(u,z)) \in \C^{N-2}$ to
\begin{equation}
\label{eqn:mdescalars}
	1+(z-(A_N(u))_{ii} + \ms{S}_i(\mathbf{m}(u,z)))m_i(u,z) = 0, \quad \im m_i(u,z) > 0, \quad i = 1, \ldots, N-2,
\end{equation}
and consider the measure with Stieltjes transform $\frac{1}{N-2}\sum_{i=1}^{N-2} m_i(u,z)$ at $z$. We claim that solutions to \eqref{eqn:mdescalars} are in bijection with solutions to \eqref{eqn:mdesn}. Indeed, since $\mc{S}_N$ maps into diagonal matrices and $A_N(u)$ is diagonal, we see directly from \eqref{eqn:mdesn} that any solution matrix $M_N(u,z)$ must be diagonal, say $M_N(u,z) = \diag(m_1(u,z), \ldots, m_{N-2}(u,z))$. Writing out \eqref{eqn:mdesn} in components, we see that $\mathbf{m}(u,z) = (m_1(u,z), \ldots, m_{N-2}(u,z))$ then also solves \eqref{eqn:mdescalars}. On the other hand, if $\mathbf{m}(u,z) = (m_1(u,z), \ldots, m_{N-2}(u,z))$ solves \eqref{eqn:mdescalars}, then $M_N(u,z) = \diag(m_1(u,z), \ldots, m_{N-2}(u,z))$ solves \eqref{eqn:mdesn}. Hence the measure induced by \eqref{eqn:mdescalars} (i.e., with Stieltjes transform $\frac{1}{N-2} \sum_{i=1}^{N-2} m_i(u,z)$ at $z$) is actually $\mu_N(u)$, the measure induced by \eqref{eqn:mdesn}. All the desired properties of this measure were already proved, locally uniformly in $u$, in Lemma \ref{lem:mde_properties}.

To check \eqref{eqn:bsg_wnleftedge}, we note that $W_N = H_N(0)$, and that $\mu_N(0)$ is supported in $[-2\sqrt{\sup_N\|\mc{S}_N\|},2\sqrt{\sup_N\|\mc{S}_N\|}]$ by \eqref{eqn:munsupport}. Then \cite[Theorem 2.4, Remark 2.5(v)]{AltErdKruNem2019} gives\footnote{
Our notation relates to the notation of \cite{AltErdKruNem2019} as follows: In their (2.1), we take $L = 1$, $\ell = 1$, and identify $\C^{1 \times 1} \otimes \C^{N \times N} \cong \C^{N \times N}$. We take all $\widetilde{a_i} = \widetilde{\beta_\nu} = \widetilde{\gamma_\nu} = 0$, $\widetilde{\alpha_1} = 1$, and $X_1 = W_N$; then their $\mathbf{X}$ is our $W_N$, and their (2.11) and (2.12) are for free. Their $s_{ij}^1$ is the same as our $s_{ij}$ above, and their $t_{ij}$'s and $y_{ij}$'s are all zero, which verifies (2.9). Their (2.10) holds since the entries of $W_{ij}$ are Gaussian with variance order $1/N$, so their Theorem 2.4 holds. Their $\rho$ is our $\mu_N(0)$; since $B \defeq [-2\sqrt{\sup_N\|\mc{S}_N\|}-1,2\sqrt{\sup_N\|\mc{S}_N\|}+1]$ contains an open superset of $\supp(\mu_N(0))$, and their $(\mathbb{D}_\epsilon)_{\epsilon > 0}$ form a nested sequence whose intersection is $\supp(\mu_N(0))$ by their Remark 2.5(v), we have $\mathbb{D}_\epsilon \subset B$ for some $\epsilon > 0$; we use this $\epsilon$ in applying their Theorem 2.4. This completes the verification of \eqref{eqn:kroneckerlocation}.
}
\begin{equation}
\label{eqn:kroneckerlocation}
    \P\left(\lambda_{\min{}}(W_N) \leq -2\sqrt{\sup_N\|\mc{S}_N\|} - 1\right) \leq \frac{C}{N^{100}}
\end{equation}
for some constant $C$, which suffices.
\end{proof}

\begin{lem}
\label{lem:bsg_rhoinftyisgood}
For each $R$ there exists $\kappa$ with
\begin{equation}
\label{eqn:bsg_wasserstein}
    \sup_{u \in B_R(0)} {\rm W}_1(\mu_N(u),\mu_\infty(u)) \leq N^{-\kappa}.
\end{equation}
Additionally, there exists $C > 0$ such that
\begin{equation}
\label{eqn:bsg_sublinear}
    \E[\abs{\det(H_N(u))}] \leq (C\max(\|u\|,1))^N.
\end{equation}
Finally, for every $R$ and $\epsilon$ we have
\begin{equation}
\label{eqn:bsg_NlogN}
    \lim_{N \to \infty} \frac{1}{N\log N} \log\left[ \sup_{u \in B_R(0)} \P(d_{\textup{BL}}(\hat{\mu}_{H_N(u)}, \mu_\infty(u)) > \epsilon) \right] = -\infty.
\end{equation}
\end{lem}
\begin{proof}
First we prove the distance estimate \eqref{eqn:bsg_wasserstein}. The general result \cite[Proposition 3.1]{BenBouMcK2022}, with $\hat{\mu}_N = \mu_\infty(u)$ and $A = 2\kappa(u)$ from \eqref{eqn:munsupport}, reduces this problem to estimating the difference between the Stieltjes transforms, specifically to showing that for every $R > 0$, there exist $\epsilon_1, \epsilon_2 > 0$ such that
\[
	\sup_{u \in B_R(0)} \frac{1}{N-2} \int_{-6\kappa(u)}^{6\kappa(u)} \abs{\Tr(M_N(u,E+\ii N^{-\epsilon_1})) - \Tr(M'_N(u,E+\ii N^{-\epsilon_1}))} \diff E \leq N^{-\epsilon_2}.
\]
Estimates of exactly this type were established in \cite[Lemma 3.1]{BenBouMcK2021II}, assuming inputs which we verified for our model as \eqref{eqn:bsg_3.1.3.3}, \eqref{eqn:bsg_sNopnorm}, \eqref{eqn:bsg_snclosesn'}, and \eqref{eqn:bsg_3.1.3.6} in Lemma \ref{lem:bsg_flatness}.

The proof of the determinant estimate \eqref{eqn:bsg_sublinear} follows \cite[Lemma 4.4]{BenBouMcK2021II}, using \eqref{eqn:bsg_3.1.3.3}. The proof of the concentration estimate \eqref{eqn:bsg_NlogN} follows \cite[Lemma 4.5]{BenBouMcK2021II}. We reproduce these proofs here: Deterministically we have
\[
	\abs{\det(H_N(u))} \leq \|H_N(u)\|^N \leq (\|W_N\| + \|A_N(u)\|)^N \leq (2\|W_N\|)^N + (2\|A_N(u)\|)^N.
\]
The $A_N(u)$ term is handled with \eqref{eqn:bsg_3.1.3.3}. For the $W_N$ term, the proof of \cite[Corollary 1.10]{BenBouMcK2022} shows that $\P(\|W_N\|) \leq e^{-cN\max(0,t-C)}$ for some constants $c, C > 0$, which implies $\E[\|W_N\|^N] \leq e^{CN}$; this checks \eqref{eqn:bsg_sublinear}. To prove the concentration result \eqref{eqn:bsg_NlogN}, we notice that the laws of the entries of $\sqrt{N}H_N(u)$ satisfy the log-Sobolev inequality with a constant that is uniform over $u$ and over the choice of entry, since they are Gaussians, and since translations of measures that satisfy the log-Sobolev inequality, also satisfy log-Sobolev inequality with the same constant. So the result \cite[Theorem 1.5]{GuiZei2000} of Guionnet and Zeitouni gives
\[
	\sup_{u \in \R^3} \P(d_{\textup{BL}}(\hat{\mu}_{H_N(u)},\E[\hat{\mu}_{H_N(u)}]) > \epsilon) \leq \frac{C_1}{\epsilon^{3/2}} \exp(-C_2N^2\epsilon^5)
\]
for some constants $C_1, C_2 > 0$. The closeness of $\E[\hat{\mu}_{H_N(u)}]$ to $\mu_N(u)$ was established in the proof of \cite[Corollary 1.10]{BenBouMcK2022}; we showed the closeness of $\mu_N(u)$ to $\mu_\infty(u)$ in \eqref{eqn:bsg_wasserstein}. 
\end{proof}

\subsection{Proofs of complexity results for total critical points.}\

The main theorem, Theorem \ref{thm:bsg}, gives results both for total critical points and for local minima. In this subsection, we prove just the former results, which are simpler. In the next subsection, we establish the additional properties of $H_N(u)$ and $\mu_\infty(u)$ needed to prove the latter results.

The following lemma is a consequence of Kac--Rice arguments found in \cite{AufChe2014}, found specifically by combining their equation (27), Lemma 2, Lemma 3, equation (36) and equation (37), in a way that we will explain after the statement of the lemma. What Auffinger and Chen write is the special case of \eqref{eqn:bsgkacrice_generalB} with $B = (-\infty,t]$, counting local minima in sublevel sets, but the more general case here is an easy variation (made by removing the indicators restricting to local minima, and/or replacing $(-\infty,t]$ with $B$, everywhere in their results). We are mostly interested in the case $B = (-\infty,t]$ or $B = \R$.

\begin{lem}
\label{lem:bsgkacrice}
With the prefactor
\[
    f(N_1, N_2) = \frac{\frac{2(\pi N_1)^{N_1/2}}{\Gamma(N_1/2)} \cdot \frac{2(\pi N_2)^{N_2/2}}{\Gamma(N_2/2)} \cdot \left(\sqrt{\frac{N}{2\pi}}\right)^3}{ \left( (2\pi N)^{N-2} \left(\frac{\xi'_1}{N_1}\right)^{N_1-1} \left(\frac{\xi'_2}{N_2}\right)^{N_2-1} \right)^{1/2}},
\]
we have, for any Borel $B \subset \R$,
\begin{align}
	\E[\Crt_N^{\textup{tot}}(B)] &= f(N_1,N_2) \int_{H_B} e^{-N\frac{\|u\|^2}{2}} \E[\abs{\det(H_N(u))}] \diff u, \notag \\
	\E[\Crt_N^{\textup{min}}(B)] &= f(N_1, N_2) \int_{H_B} e^{-N\frac{\|u\|^2}{2}} \E[\abs{\det(H_N(u))} \mathbf{1}_{H_N(u) \geq 0}] \diff u. \label{eqn:bsgkacrice_generalB}
\end{align}
\end{lem}

Auffinger and Chen write this result in the form 
\[
	\E[\Crt_N^{\textup{tot}}(t)] = f(N_1,N_2) \frac{2\pi}{N} \int_{-\infty}^t \E[\abs{\det \nabla^2 \mc{H}_N(\mathbf{n})} | \mc{H}_N(\mathbf{n}) = Nx] e^{-N\frac{x^2}{2}} \diff x
\]
(see their eq. (36); although to streamline, we consider the variant that counts total critical points), where $\mathbf{n}$ is the special ``double north pole'' $\mathbf{n} = (0,\ldots,0,\sqrt{N_1},0,\ldots,0,\sqrt{N_2})$. Their Lemma 3 computes the distribution of the Hessian $\nabla^2 \mc{H}_N(\mathbf{n})$, conditioned on $\mc{H}_N(\mathbf{n})$, but some of their formulas are slightly incorrect as written: Some formulas $\E[AB]$ should in fact be $\Cov(A,B)$, and every $\xi'_1\xi'_2$ should in fact be $\xi''_{12}$, except for the penultimate one, which should be $\xi''_{12} - \xi'_1\xi'_2$. The corrected formulas from point 7 of their Lemma 3 are as follows, written for all $(i,i',j,j') \in I_1 \times I_1 \times I_2 \times I_2$ (their $H$ is our $\mc{H}_N$):
\begin{align*}
    \E[\nabla^2 H^{11}_{ii}(\mathbf{n}) | H(\mathbf{n}) = Nu_0] &= -\frac{N\xi'_1}{N_1}u_0, \\
    \E[\nabla^2 H^{22}_{jj}(\mathbf{n}) | H(\mathbf{n}) = Nu_0] &= -\frac{N\xi'_2}{N_2}u_0, \\
    \Var(\nabla^2 H^{11}_{ii'}(\mathbf{n}) | H(\mathbf{n}) = Nu_0] &= \frac{N}{N_1^2} \xi''_1 \quad \text{when} \quad i \neq i', \\
    \Var(\nabla^2 H^{22}_{jj'}(\mathbf{n}) | H(\mathbf{n}) = Nu_0] &= \frac{N}{N_2^2} \xi''_2 \quad \text{when} \quad j \neq j', \\
    \Var(\nabla^2 H^{12}_{ij}(\mathbf{n}) | H(\mathbf{n}) = Nu_0) &= \Var(\nabla^2 H^{21}_{ji}(\mathbf{n}) | H(\mathbf{n}) = Nu_0) = \frac{N}{N_1N_2} \xi''_{12}, \\
    \Cov(\nabla^2 H^{11}_{ii}(\mathbf{n}), \nabla^2 H^{11}_{i'i'}(\mathbf{n}) | H(\mathbf{n}) = Nu_0) &= \frac{N}{N_1^2}(\xi'_1 + (1+2\delta_{ii'})\xi''_1 - (\xi'_1)^2), \\
    \Cov(\nabla^2 H^{22}_{jj}(\mathbf{n}), \nabla^2 H^{22}_{j'j'}(\mathbf{n}) | H(\mathbf{n}) = Nu_0) &= \frac{N}{N_2^2}(\xi'_2 + (1+2\delta_{jj'})\xi''_2 - (\xi'_2)^2), \\
    \Cov(\nabla^2 H^{11}_{ii}(\mathbf{n}), \nabla^2 H^{22}_{jj}(\mathbf{n}) | H(\mathbf{n}) = Nu_0) &= \frac{N}{N_1N_2}(\xi''_{12} - \xi'_1\xi'_2), \\
    \Cov(\nabla^2 H^{11}_{ii}(\mathbf{n}), \nabla^2 H^{22}_{jj}(\mathbf{n}) | H(\mathbf{n}) = Nu_0) &= \frac{N}{N_1N_2}(\xi''_{12} - \xi'_1\xi'_2).
\end{align*}
Furthermore, in their eq. (34), the Gaussian corrections $Z_i$ should in general be correlated with each other (although still independent of everything else). More precisely, to compute the correlation, we see that if the conditioned Hessian has the form
\[
    T = W_N + \begin{pmatrix} \frac{\sqrt{N}}{N_1}\alpha_1 Z_1 \Id_{N_1} & 0 \\ 0 & \frac{\sqrt{N}}{N_2}\alpha_2 Z_2 \Id_{N_2} \end{pmatrix} - \begin{pmatrix} \frac{N\xi'_1}{N_1} u_0 \Id_{N_1} & 0 \\ 0 & \frac{N\xi'_2}{N_2} u_0 \Id_{N_2} \end{pmatrix}, \quad \begin{pmatrix} Z_1 \\ Z_2 \end{pmatrix} \sim \mc{N}\left(0, \begin{pmatrix} 1 & \eta \\ \eta & 1 \end{pmatrix} \right)
\]
for some $\eta$, where the Gaussians $Z_i$ are independent of everything else, then 
\[
    \frac{N}{N_1N_2}(\xi''_{12} - \xi'_1\xi'_2) = \Cov(T^{11}_{ii},T^{22}_{jj}) = \Cov\left(\frac{\sqrt{N}}{N_1}\alpha_1Z_1, \frac{\sqrt{N}}{N}\alpha_2Z_2\right) = \frac{N}{N_1N_2}\alpha_1\alpha_2\eta, 
\]
and thus we should take $\eta = \sigma$, where $\sigma$ is defined in \eqref{eqn:def_sigma}. (We normalize the GOE differently from Auffinger and Chen, but change factors elsewhere so that, in our $W_N = (\begin{smallmatrix} G_1 & G \\ G^T & G_2 \end{smallmatrix})$, our $G_i$ is ultimately the same as their $(\frac{N(N_i-1)}{N_i^2}2\xi''_i)^{1/2}M^{N_i-1}$ for $i = 1, 2$.) Lemma \ref{lem:sigma_def} ensures that $(\begin{smallmatrix} 1 & \sigma \\ \sigma & 1 \end{smallmatrix})$ is indeed a covariance matrix. We could track these correlations -- this would lead to some quadratic form $\ip{u,Au}$ in \eqref{eqn:s_bsg} rather than $\|u\|^2$ -- but we prefer to work with independent variables, so we write
\[
    \begin{pmatrix} Z_1 \\ Z_2 \end{pmatrix} \overset{d}{=} \begin{pmatrix} 1 & \sigma \\ \sigma & 1 \end{pmatrix}^{1/2} \begin{pmatrix} Y_1 \\ Y_2 \end{pmatrix} = \frac{1}{2} \begin{pmatrix} \sigma_+ & \sigma_- \\ \sigma_- & \sigma_+ \end{pmatrix} \begin{pmatrix} Y_1 \\ Y_2 \end{pmatrix}
\]
for independent standard Gaussian variables $Y_1, Y_2$. Ultimately we prefer our Gaussians to have variance $\frac{1}{N}$; with these corrections, conditioned on $\mc{H}_N(\mathbf{n}) = Nu_0$, the Hessian $\nabla^2 \mc{H}_N(\mathbf{n})$ has the same distribution as what we would call in our notation $W_N + A_N(u_0,U_1,U_2)$, where $U_1$ and $U_2$ are Gaussian variables of variance $\frac{1}{N}$, independent of each other and of everything else. In short, they write a one-dimensional integral consisting of matrices that have correlations along the diagonal, where the correlations are induced by two underlying random variables that are independent of everything else. We prefer to integrate over these variables separately, to get a three-dimensional integral consisting of matrices with no correlations, where $x$ corresponds to $u_0$ and $U_i$ corresponds to $u_i$ for $i = 1, 2$ (the Gaussian integrals cancel the extra $\frac{2\pi}{N}$ term). This is how we arrive at Lemma \ref{lem:bsgkacrice}.

\begin{proof}[Proof of Lemma \ref{lem:scontinuous} and \eqref{eqn:bsg_tot}]
Since
\begin{equation}
\label{eqn:fn1n2limit}
    \lim_{N \to \infty} \frac{1}{N} \log f(N_1,N_2) = \frac{1+\gamma\log\left(\frac{\gamma}{\xi'_1}\right) + (1-\gamma)\log\left(\frac{1-\gamma}{\xi'_2}\right)}{2},
\end{equation}
it remains only to understand the integrals appearing in Lemma \ref{lem:bsgkacrice}.

Specifically, to show \eqref{eqn:bsg_tot} it suffices to show
\begin{equation}
\label{eqn:abovethm4.1}
\begin{split}
    \lim_{N \to \infty} \frac{1}{N} \log \int_{H_t} e^{-N\frac{\|u\|^2}{2}} \E[\abs{\det(H_N(u))}] \diff u = \sup_{u \in H_t} \mc{S}_{\textup{bsg}}[u], \\
    \lim_{N \to \infty} \frac{1}{N} \log \int_{\R^3} e^{-N\frac{\|u\|^2}{2}} \E[\abs{\det(H_N(u))}] \diff u = \sup_{u \in \R^3} \mc{S}_{\textup{bsg}}[u].
\end{split}
\end{equation}
We wish to apply Theorem 4.1 in \cite{BenBouMcK2022}  with the choices $\alpha = 1/2$, $p = 2$ (recall $N = (N-2)+2)$ is two more than the size of $H_N(u)$), and $\mf{D} = \R^3$ or $\mf{D} = H_t$. This theorem has three hypotheses that we need to verify in our case. The first (``Assumptions locally uniform in $u$'') was checked in Lemma \ref{lem:bsg_checkingassnsK}; the second (``Limit measures'') was checked in \eqref{eqn:bsg_wasserstein}; the third (``Continuity and decay in $u$'') was checked in \eqref{eqn:bsg_sublinear}. In fact, these three hypotheses also imply Lemma \ref{lem:scontinuous}, as shown in \cite[Lemma 4.4]{BenBouMcK2022}.
\end{proof}

\subsection{Proofs of complexity results for local minima.}\

In this subsection, we establish some additional properties of $H_N(u)$ and $\mu_\infty(u)$ and finish the proof of Theorem \ref{thm:bsg}.

\begin{lem}
\label{lem:bsg_nooutliers}
For every $\epsilon > 0$ and $R > 0$, we have
\begin{equation}
\label{eqn:bsg_nooutliers}
    \lim_{N \to \infty} \inf_{u \in B_R(0)} \P(\Spec(H_N(u)) \subset [\ell(\mu_\infty(u))-\epsilon, r(\mu_\infty(u)) + \epsilon]) = 1
\end{equation}
and in fact the extreme eigenvalues of $H_N(u)$ converge in probability to the endpoints of $\mu_\infty(u)$.
\end{lem}
\begin{proof}
The local law of Alt \emph{et al.} \cite{AltErdKruNem2019} is written with respect to a Dyson equation over $\C^{N-2}$. We already showed, in the proof of Lemma \ref{lem:bsg_checkingassnsK}, that the measures induced by this Dyson equation for the matrices $H_N(u)$ are actually $\mu_N(u)$. In the same way, one can show that the measures induced by this Dyson equation for the matrices $H'_N(u)$ are actually $\mu_\infty(u)$; this is why we introduced those matrices. The rest of the argument is exactly as in the proof of \cite[Lemma 4.6]{BenBouMcK2021II}: it uses the local law \cite{AltErdKruNem2019} to localize the spectrum of $H'_N(u)$ near the support of $\mu_\infty(u)$, then Lemma \ref{lem:bsg_hn_vs_hprimen} to relate $H_N(u)$ to $H'_N(u)$, and finally \eqref{eqn:bsg_NlogN} to show that the extreme eigenvalues of $H_N(u)$ do not push inside the support of $\mu_\infty(u)$.
\end{proof}

\begin{lem}
\label{lem:bsg_topology}
With
\[
	\mc{G}_{+\epsilon} = \{u \in \R^3 : \ell(\mu_\infty(u)) \geq 2\epsilon\}
\]
(this is the same definition as \cite[(4-5)]{BenBouMcK2022}) and $\mc{G}$ as defined in \eqref{eqn:bsg_defG}, we have that each $\mc{G}_{+\epsilon}$ is convex, that $\mc{G}$ is convex and closed, that $\mc{G}_{+1} \cap H_t$ has positive measure for every $t$, and that 
\[
    \overline{\bigcup_{\epsilon > 0} \mc{G}_{+\epsilon}} = \mc{G} \quad \text{and} \quad \overline{H_t \cap \left( \bigcup_{\epsilon > 0} \mc{G}_{+\epsilon} \right)} = H_t \cap \mc{G} \neq \emptyset \quad \text{ for all $t$}.
\]
\end{lem}
\begin{proof}
Convexity for $\mc{G}_{+\epsilon}$ (and for $\mc{G}$, which is the same as $\mc{G}_{+\epsilon}$ with $\epsilon = 0$) is proved exactly as in \cite[Lemma 4.7]{BenBouMcK2021II}. 

For simplicity, we restrict ourselves to $u$ in the quarter space
\[
    \mc{Q} = \{(u_0,u_1,u_2) : u_1 \geq 0, u_2 \geq 0\}.
\]
For $u \in \mc{Q}$ we have
\begin{align*}
    \lambda_{\min{}}(A_N(u)) &= \min\left\{\frac{N}{N_1}\left(\frac{\alpha_1\sigma_+}{2}u_1 + \frac{\alpha_1\sigma_-}{2}u_2 - \xi'_1 u_0\right), \frac{N}{N_2}\left(\frac{\alpha_2\sigma_-}{2}u_1 + \frac{\alpha_2\sigma_+}{2} u_2 - \xi'_2 u_0\right) \right\} \\
    &\geq u_0 \min\left\{-\frac{2\xi'_1}{\gamma}, -\frac{2\xi'_2}{1-\gamma}\right\} \eqdef -\kappa_{\textup{bsg}} u_0.
\end{align*}
Combining this with \eqref{eqn:bsg_wnleftedge}, we find
\[
    \liminf_{N \to \infty} \lambda_{\min{}}(H_N(u)) \geq -\kappa_{\textup{bsg}} u_0 - 2\sqrt{\sup_N \|\mc{S}_N\|}-1
\]
for $u \in \mc{Q}$. Since we know from Lemma \ref{lem:bsg_nooutliers} that $\lambda_{\min{}}(H_N(u))$ tends to $\ell(\mu_\infty(u))$ in probability, we find that $\ell(\mu_\infty(u)) \geq -\kappa_{\textup{bsg}}u_0 - 2\sqrt{\sup_N \|\mc{S}_N\|} - 1$ for all $u \in \mc{Q}$. Hence
\begin{equation}
\label{eqn:subsetofg}
	 \left\{(u_0,u_1,u_2) \in \R^3 : u_0 \leq \min\left(t,-\frac{2\epsilon+2\sqrt{\sup_N \|\mc{S}_N\|}+1}{\kappa_{\textup{bsg}}}\right), u_1 \geq 0, u_2 \geq 0\right\} \subset (\mc{G}_{+\epsilon} \cap H_t) \subset \mc{G}
\end{equation}
for all $\epsilon$ and all $t$, which shows that $\mc{G}_{+1} \cap H_t$ has positive measure for all $t$, and that $H_t \cap \mc{G} \neq \emptyset$ for all $t$, since $\mc{G}$ is the same as $\mc{G}_{+\epsilon}$ with $\epsilon = 0$.

Finally, we note that the inclusion $\cup_{\epsilon > 0} \mc{G}_{+\epsilon} \subset \mc{G}$ is clear, and that $\mc{G}$ is closed by \cite[Lemma 4.6]{BenBouMcK2022}. (Applying this lemma requires (a subset of) the assumptions of Theorem 4.1 in \cite{BenBouMcK2022}, but we checked all the assumptions of Theorem 4.1 in the paragraph just below \eqref{eqn:abovethm4.1}.)

To show the reverse inclusion, write $e_1 = (1,0,0)$; then for $\delta > 0$ we have $A_N(u-\delta e_1) \geq A_N(u) + \kappa'_{\textup{bsg}}\delta\Id$, with $\kappa'_{\textup{bsg}} = \min\{2\xi'_1/\gamma,2\xi'_2/(1-\gamma)\}$, so that by the convergence in probability of Lemma \ref{lem:bsg_nooutliers} we have $\ell(\mu_\infty(u-\delta e_1)) \geq \ell(\mu_\infty(u)) + \kappa'_{\textup{bsg}}\delta$. This completes the proof of the equality $\overline{\cup_\epsilon \mc{G}_{+\epsilon}} = \mc{G}$. The version intersected with $H_t$ is an exercise in point-set topology, since $\cup_\epsilon \mc{G}_{+\epsilon}$ is convex as a union of nested convex sets, $H_t$ is a half-space, and their intersection has non-empty interior by the arguments above.
\end{proof}

\begin{proof}[Proof of Theorem \ref{thm:bsg}]
It remains only to show the results on local minima. Applying Lemma \ref{lem:bsgkacrice} and \eqref{eqn:fn1n2limit}, we see that to show \eqref{eqn:bsg_min} it suffices to show
\begin{align*}
    \lim_{N \to \infty} \frac{1}{N} \log \int_{H_t} e^{-N\frac{\|u\|^2}{2}} \E[\abs{\det(H_N(u))} \mathbf{1}_{H_N(u)} \geq 0] \diff u &= \sup_{u \in \mc{G} \cap \mc{H}_t} \mc{S}_{\textup{bsg}}[u], \\
    \lim_{N \to \infty} \frac{1}{N} \log \int_{\R^3} e^{-N\frac{\|u\|^2}{2}} \E[\abs{\det(H_N(u))} \mathbf{1}_{H_N(u)} \geq 0] \diff u &= \sup_{u \in \mc{G}} \mc{S}_{\textup{bsg}}[u].
\end{align*}
Here we will apply Theorem 4.5 in \cite{BenBouMcK2022}, with the same choices of parameters as in Theorem 4.1. This theorem requires three hypotheses beyond those of Theorem 4.1: The first (``Superexponential concentration'') was checked in \eqref{eqn:bsg_NlogN}; the second (``No outliers'') was checked in \eqref{eqn:bsg_nooutliers}; the third (``Topology'') was checked in Lemma \ref{lem:bsg_topology}.
\end{proof}

\subsection{Proofs of corollaries about pure models.}\

\begin{proof}[Proof of Corollary \ref{cor:einfty}]
Directly from the MDE \eqref{eqn:bsg_scalars}, we obtain the symmetry
\[
    \mu_\infty(-u,\lambda) = \mu_\infty(u,-\lambda).
\]
In particular, $\mu_\infty(0,\lambda)$ is an even function of $\lambda$, so its left edge is strictly negative, hence $u = 0$ is not an element of $\mc{G}_{\textup{pure}}$. Since $\mc{G}_{\textup{pure}}$ has the form $(-\infty,-E_\infty(p,q,\gamma)]$ we conclude $-E_\infty(p,q,\gamma) < 0$. Since $\Sigma^{\textup{min}}(t) = \text{constant} + \sup_{u \in \mc{G}_{\textup{pure}} \cap (-\infty,t]} \mc{S}_{\textup{bsg}}[u]$ from \eqref{eqn:bsg_min}, we conclude that $\Sigma^{\textup{min}}(t)$ stabilizes at $t = -E_\infty(p,q,\gamma)$.

For each $\epsilon$, consider the half-space
\[
    \widetilde{H}_\epsilon = \{(u_0,u_1,u_2) \in \R^3 : u_0 \geq -E_\infty(p,q,\gamma)+\epsilon\}.
\]
By Markov's inequality and Lemma \ref{lem:bsgkacrice}, we have
\begin{align*}
    \P(\Crt^{\textup{min}}_N((-E_\infty(p,q,\gamma)+\epsilon,\infty)) \geq 1) &\leq \E[\Crt^{\textup{min}}_N((-E_\infty(p,q,\gamma)+\epsilon,\infty))]] \\
    &= f(N_1,N_2) \int_{\widetilde{H}_\epsilon} e^{-N\frac{\|u\|^2}{2}} \E[\abs{\det(H_N(u))} \mathbf{1}_{H_N(u) \geq 0}] \diff u.
\end{align*}
In the companion paper \cite[(4-5)]{BenBouMcK2022} we considered a sequence of nested sets $\mc{G}_{-\delta}$ defined by
\[
    \mc{G}_{-\delta} = \{u \in \R^3 : \mu_\infty(u)((-\infty,-\delta)) \leq \delta\}.
\]
Since $\mu_\infty(u)$ depends only on $u_0$ in the pure case we are currently considering, each $\mc{G}_{-\delta}$ is of the form $\mc{G}_{-\delta} = \{u_0 \times \R^2 : u_0 \in \mc{G}_{\textup{pure}, -\delta}\}$ for some set $\mc{G}_{\textup{pure}, -\delta} \subset \R$. In fact, we claim that $\mc{G}_{\textup{pure},-\delta}$ is an interval of the form
\begin{equation}
\label{eqn:gminusdelta_convex}
    \mc{G}_{\textup{pure}, -\delta} = (-\infty,f(\delta)].
\end{equation}
Assume this claim momentarily. From the definitions one can see that $\cap_{\delta > 0} \mc{G}_{-\delta} = \mc{G}$, and thus \eqref{eqn:gpure} tells us that $\lim_{\delta \downarrow 0} f(\delta) = -E_\infty(p,q,\gamma)$. Hence there exists a small $\delta = \delta(\epsilon) > 0$ with $f(\delta) < -E_\infty(p,q,\gamma) + \epsilon$. For this $\delta$ we therefore have $\widetilde{H}_{\epsilon} \subset (\mc{G}_{-\delta})^c$, but we showed in \cite[Lemma 4.7]{BenBouMcK2022} that
\[
    \lim_{N \to \infty} \frac{1}{N}\log \int_{(\mc{G}_{-\delta})^c} e^{-N\frac{\|u\|^2}{2}} \E[\abs{\det(H_N(u))} \mathbf{1}_{H_N(u) \geq 0}] \diff u = -\infty
\]
for every $\delta > 0$. This completes the proof, modulo \eqref{eqn:gminusdelta_convex}.

Now we prove \eqref{eqn:gminusdelta_convex}. Since $\mu_\infty((u_0,u_1,u_2))$ depends on $u_0$ only, we abuse notation and write $\mu_\infty(u_0)$. Notice that $\mu_\infty(u_0)$ is the limiting empirical measure of the random matrix $W_N + u_0B_N$, where 
\[
    B_N = A_N(1,0,0) = -\left(\begin{smallmatrix} \frac{N\xi'_1}{N_1} & 0 \\ 0 & \frac{N\xi'_2}{N_2} \end{smallmatrix}\right)
\]
has strictly negative eigenvalues. The Courant--Fischer variational characterization of eigenvalues gives that, for each $i \in \llbracket 1, N \rrbracket$, the $i$th eigenvalue of $W_N + u_0B_N$ is a non-increasing function of $u_0$. Thus
\[
    \frac{1}{N}\#\{i : \lambda_i(W_N + u_0B_N) < -\delta\}
\]
is almost surely non-decreasing in $u_0$, hence its $N\to+\infty$ limit $\mu_\infty(u_0)((-\infty,-\delta))$ is also non-decreasing in $u_0$. This shows that $\mc{G}_{-\delta}$ is a single interval containing arbitrarily large negative values. From its definition and continuity of the map $u \mapsto \mu_\infty(u)$ (see the proof of \cite[Lemma 4.6]{BenBouMcK2022}, which uses \eqref{eqn:bsg_wasserstein}) one can see that it is closed, which completes the proof of \eqref{eqn:gminusdelta_convex}.
\end{proof}

\begin{proof}[Proof of Corollary \ref{cor:explicit}]
Since $u_1$ and $u_2$ play no role, we drop them from the notation. The scalar problem \eqref{eqn:bsg_scalars} is solved by $m_1(u_0,z) = m_2(u_0,z) = m(u_0,z)$, which satisfies a quadratic equation given by \eqref{eqn:bsg_scalars}; this yields
\[
    \mu_\infty(u_0, \diff\lambda) = \frac{\sqrt{(4(p+q-1)(p+q) - (\lambda+(p+q)u_0)^2)_+}}{2\pi(p+q-1)(p+q)} \diff \lambda.
\]
Since the left edge of this measure is explicit, $E_\infty(p,q,\frac{p}{p+q})$ can be computed directly. After changing variables we obtain
\[
    \mc{S}_{\textup{bsg}}[(u_0,0,0)] = \log\left(\sqrt{(p+q-1)(p+q)}\right) + \Omega\left(u_0\sqrt{\frac{p+q}{p+q-1}}\right)-\frac{(u_0)^2}{2}
\]
which is even, strictly concave, and uniquely maximized at zero, since $u \mapsto \Omega(su)-\frac{u^2}{2}$ has this property for $\abs{s} < \sqrt{2}$, and here $\sqrt{\frac{p+q}{p+q-1}} < \sqrt{2}$. This allows us to solve the variational problem.
\end{proof}

\begin{proof}[Proof of Corollary \ref{cor:groundstate}]
To locate the ground state, note that if $\min_{u,v} \mc{H}_{N,p,q}(u,v) \leq N(-E_0(p+q)-\epsilon)$, then $\Crt_N^{\text{tot}}(-E_0(p+q)-\epsilon) \geq 1$; then apply Markov's inequality. To estimate $-E_0(p+q)$, use the crude bounds 
\[
	0 \leq \Omega\left(t\sqrt{\frac{p+q}{p+q-1}}\right) \leq -t\sqrt{\frac{p+q}{p+q-1}},
\]
valid for all $t \leq -\sqrt{2}$, to get upper and lower bounds for $\Sigma_{p+q}(t)$ and hence for $-E_0(p+q)$.
\end{proof}

\begin{rem}
\label{rem:pure1q}
We remark briefly on the restriction $\xi''_1 > 0$ and $\xi''_2 > 0$, which is equivalent to ``neither a linear combination of pure $(1,q)$ spins for different $q$ values, nor a linear combination of pure $(p,1)$ spins for different $p$ values.'' In order to apply our Laplace-method arguments, we need the measures $\mu_\infty(u)$ to admit densities. We know of two strategies to show that a measure induced by the Dyson equation admits a density: Either check that the self-energy operator $\mc{S}$ in the Dyson equation is flat and import results of \cite{AjaErdKru2019} (which is our strategy here), or ``by hand,'' meaning manipulate the Dyson equation in a clever way to show that the solution $M_N(u,z)$ satisfies $\sup_{N,u,z} \|M_N(u,z)\| < \infty$ (which is our strategy for the ``elastic-manifold'' model in \cite{BenBouMcK2021II}). 

If $\min(\xi''_1, \xi''_2) = 0$, the self-energy operators $\mc{S}_N$ and $\mc{S}'_N$ are not flat: For example, if $\xi''_1 = 0$, then
\[
    \mc{S}_N\left[ \begin{pmatrix} T_{11} & 0 \\ 0 & 0 \end{pmatrix} \right] = \begin{pmatrix} 0 & 0 \\ 0 & \frac{N\xi''_{12}}{N_1N_2} \Tr(T_{11}) \Id + \frac{N\xi''_2}{N_2^2} \diag(T_{11}) \end{pmatrix} \not\geq \frac{1}{\kappa(N-2)}\Tr(T).
\]
The missing piece is thus to establish regularity ``by hand,'' which we do not know how to do for this model.
\end{rem}

\addcontentsline{toc}{section}{References}
\bibliographystyle{alpha-abbrvsort}
\bibliography{complexitybib}

\end{document}